\documentclass[a4paper,10pt]{amsart}
\usepackage{pb-diagram}
\usepackage{xypic}
\usepackage{amscd}
\usepackage{bm}
\usepackage{slashbox}
\usepackage{amsmath,amssymb,amscd,bbm,bbding,amsthm,mathrsfs,hyperref}
\usepackage{latexsym}
\usepackage{amsfonts}
\usepackage{amsthm}
\usepackage{indentfirst}
 \newtheorem{thm}{Theorem}[section]
 
 \newtheorem{lem}[thm]{Lemma}
 \newtheorem{prop}[thm]{Proposition}
 \newtheorem{defn}[thm]{Definition}
 
 \newtheorem{rem}[thm]{Remark}

 \newcommand{\Hom}{\mathrm{Hom}}

\title{Homological properties of $3$-dimensional DG Sklyanin algebras}
\author{X.-F. Mao}
\address{Department of Mathematics, Shanghai University, Shanghai 200444, China}
\email{xuefengmao@shu.edu.cn}

\author{H. Wang}
\address{Department of Mathematics, Shanghai University, Shanghai 200444, China}
\email{happywang97@shu.edu.cn}

\author{X.-T. Wang}
\address{Department of Mathematics, Howard University, Washington DC, 20059, USA}
\email{xingting.wang@Howard.edu}

\author{Y.-N. Yang}
\address{Department of Mathematics, Shanghai University, Shanghai 200444, China}
\email{mooly@shu.edu.cn}

\author{M.-Y.Zhang}
\address{Department of Mathematics, Shanghai University, Shanghai 200444, China}
\email{zmy1023@shu.edu.cn}

\date{}

\begin{document}

\begin{abstract}
In this paper, we introduce
 the notion of DG Sklyanin algebras, which are connected cochain DG algebras whose
underlying graded algebras are Sklyanin algebras. Let $\mathcal{A}$ be a $3$-dimensional DG Sklyanin algebra with $\mathcal{A}^{\#}=S_{a,b,c}$, where $(a,b,c)\in \Bbb{P}_k^2-\mathfrak{D}$ and
$$\mathfrak{D}=\{(1,0,0), (0,1,0),(0,0,1)\}\sqcup\{(a,b,c)|a^3=b^3=c^3\}.$$
We systematically study its differential structures and various homological properties.
Especially, we figure out the conditions for $\mathcal{A}$ to be Calabi-Yau, Koszul, Gorenstein and homologically smooth, respectively.

\end{abstract}

\maketitle

\section{introduction}
The theory of differential graded algebras (DG algebras, for short) and their modules has numerous applications in rational homotopy theory as well as algebraic geometry. In particular, general results in DG homological algebra depend on the constructions of some interesting families of DG algebras. In the literature, there has been many papers on graded commutative DG algebras. Especially, the Sullivan algebra and De Rham complex are fundamental DG algebra models in rational homotopy theory and differential geometry, respectively.
Comparatively speaking, less attention has been paid to non-commutative DG algebras. To change this situation, many attempts have been made to construct some interesting family of non-commutative cochain DG algebras with some nice homological properties such as homologically smoothness, Gorensteinness and Calabi-Yau property. In \cite{MHLX}, \cite{MGYC} and \cite{MXYA}, DG down-up algebras,  DG polynomial algebras and DG free algebras are introduced and  systematically studied, respectively. It is exciting to discover that non-trivial DG down-up algebras, some DG polynomial algebras and DG free algebras with $2$ degree $1$ variables are Calabi-Yau DG algebras. Since Ginzburg introduced Calabi-Yau (DG) algebras in \cite{Gin}, they have been extensively studied due to their links to
mathematical physics, representation theory  and non-commutative algebraic geometry. In general, the homological properties of a DG algebra are determined by the joint effects of its underlying graded algebra structure and differential structure. Although there have been some discriminating methods (cf.\cite{HM,MYY}), it is still difficult in general to  detect the Calabi-Yau property of a cochain DG algebra. Those newly discovered examples of Calabi-Yau DG algebras among DG down-up algebras, DG polynomial algebras and DG free algebras inspire us to construct cochain DG algebras on some well-known Artin-Schelter regular algebras.

The $3$-dimensional Sklyanin algebras form the most important class of Artin-Schelter regular algebras of global dimension $3$. Let $k$ be an algebraically closed field of characteristic $0$ and $\mathfrak{D}$ the subset of the projective plane $\Bbb{P}_k^2$ consisting of the $12$ points:
$\mathfrak{D}:=\{(1,0,0), (0,1,0),(0,0,1)\}\sqcup\{(a,b,c)|a^3=b^3=c^3\}.$
Recall that the points $(a,b,c)\in \Bbb{P}_k^2-\mathfrak{D}$ parametrize the $3$-dimensional Sklyanin algebras,
$$S_{a,b,c}=\frac{k\langle x,y,z\rangle}{(f_1,f_2,f_3)},$$
where \begin{align*} f_1&=ayz+bzy+cx^2\\
                     f_2&=azx+bxz+cy^2\\
                     f_3&=axy+byx+cz^2.
                     \end{align*}
We say that a cochain DG algebra $\mathcal{A}$ is a $3$-dimensional Sklyanin DG algebra if its underlying graded algebra $\mathcal{A}^{\#}$ is a $3$-dimensional Sklyanin algebra $S_{a,b,c}$, for some $(a,b,c)\in \Bbb{P}_k^2-\mathfrak{D}$.
We describe all possible differential structures on a $3$-dimensional Sklyanin DG algebra by the following theorem (cf.Theorem \ref{diffstr}):
 \\
\begin{bfseries}
Theorem\,A.
\end{bfseries}
Let $\mathcal{A}$ be a $3$-dimensional DG Sklyanin algebra with $\mathcal{A}^{\#}=S_{a,b,c}$, $(a,b,c)\in \Bbb{P}_k^2-\mathfrak{D}$.
Then we have the following statements:

(1)$\partial_{\mathcal{A}}=0$ if either $|a|\neq |b|$ or $c\neq 0$.

(2)$\partial_{\mathcal{A}}$ is defined by
\begin{align*}
\begin{cases}
\partial_{\mathcal{A}}(x)=\alpha x^2+\beta xy +\gamma xz\\
\partial_{\mathcal{A}}(y)=\alpha yx +\beta y^2+\gamma yz\\
\partial_{\mathcal{A}}(z)=\alpha xz +\beta yz +\gamma z^2,\, \text{for some}\,(\alpha,\beta,\gamma)\in \Bbb{A}_k^3, \,\text{if}\,a=-b, c=0.
\end{cases}
\end{align*}

(3)$\partial_{\mathcal{A}}$ is defined by
\begin{align*}
\left(
                         \begin{array}{c}
                           \partial_{\mathcal{A}}(x)\\
                           \partial_{\mathcal{A}}(y)\\
                           \partial_{\mathcal{A}}(z)
                         \end{array}
                       \right)=M\left(
                         \begin{array}{c}
                           x^2\\
                           y^2\\
                           z^2
                         \end{array}
                       \right), \text{for some}\, M\in M_3(k), \,\text{if}\,a=b, c=0.
\end{align*}

The motivation of this paper is to study the various homomological properties of $3$-dimensional DG Sklyanin algebras.
 Recall that a square matrix is called a quasi-permutation matrix if each row and each column has at most one non-zero element, and $\mathrm{QPL}_3(k)$ is the subgroup of $\mathrm{GL}_3(k)$ consisting of quasi-permutation matrices.
 We have the following theorem (cf. Proposition \ref{firstcase},Proposition \ref{polydg} and Proposition \ref{quancase}).
\\
\begin{bfseries}
Theorem\,B.
\end{bfseries}
Let $\mathcal{A}$ be a $3$-dimensional DG Sklyanin algebra with $\mathcal{A}^{\#}=S_{a,b,c}$, $(a,b,c)\in \Bbb{P}_k^2-\mathfrak{D}$.
Then we have the following statements.
\begin{enumerate}
\item If either $|a|\neq |b|$ or $c\neq 0$, then $\mathcal{A}$ is a Koszul Calabi-Yau DG algebra with zero differential.
\item If $a=-b, c=0$, then $\mathcal{A}$ is a Koszul Calabi-Yau DG algebra when $\partial_{\mathcal{A}}=0$, and $\mathcal{A}$ is Gorenstein but neither Koszul nor homologically smooth when $\partial_{\mathcal{A}}\neq 0$.
\item If $a=b, c=0$, and $\partial_{\mathcal{A}}$ is determined by a matrix $M\in M_3(k)$ as in Theorem A, then $\mathcal{A}$ is Koszul, and it is  not Calabi-Yau (or not homologically smooth, or not Gorenstein) if and only if there exists $C=(c_{ij})_{3\times 3}\in \mathrm{QPL}_3(k)$ satisfying $M=C^{-1}N(c_{ij}^2)_{3\times 3}$,
where $$N=\left(
                                 \begin{array}{ccc}
                                   1 & 1 & 0 \\
                                   1 & 1 & 0 \\
                                   1 & 1 & 0 \\
                                 \end{array}
                               \right)\,\,\,\, \text{or}\,\,\,\, N=\left(
                                 \begin{array}{ccc}
                                   n_{11} & n_{12} & n_{13} \\
                                   l_1n_{11} & l_1n_{12} & l_1n_{13} \\
                                   l_2n_{11} & l_2n_{12} & l_2n_{13} \\
                                 \end{array}
                               \right)$$
                            with  $n_{12}l_1^2+n_{13}l_2^2\neq n_{11}, l_1l_2\neq 0$ and $4n_{12}n_{13}l_1^2l_2^2= (n_{12}l_1^2+n_{13}l_2^2-n_{11})^2$.
\end{enumerate}

----------------------------------------------------------------------
\section{Notations and conventions}
Throughout this paper, $k$ is an algebraically closed field of characteristic $0$. For any $k$-vector space $V$, we write $V^*=\Hom_{k}(V,k)$. Let $\{e_i|i\in I\}$ be a basis of a finite dimensional $k$-vector space $V$.  We denote the dual basis of $V$ by $\{e_i^*|i\in I\}$, i.e., $\{e_i^*|i\in I\}$ is a basis of $V^*$ such that $e_i^*(e_j)=\delta_{i,j}$. For any graded vector space $W$ and $j\in\Bbb{Z}$, the $j$-th suspension $\Sigma^j W$ of $W$ is a graded vector space defined by $(\Sigma^j W)^i=W^{i+j}$.

A cochain DG algebra is
a graded
$k$-algebra $\mathcal{A}$ together with a differential $\partial_{\mathcal{A}}: \mathcal{A}\to \mathcal{A}$  of
degree $1$ such that
\begin{align*}
\partial_{\mathcal{A}}(ab) = (\partial_{\mathcal{A}} a)b + (-1)^{|a|}a(\partial_{\mathcal{A}} b)
\end{align*}
for all graded elements $a, b\in \mathcal{A}$.
For any DG algebra $\mathcal{A}$,  we denote $\mathcal{A}\!^{op}$ as its opposite DG
algebra, whose multiplication is defined as
 $a \cdot b = (-1)^{|a|\cdot|b|}ba$ for all
graded elements $a$ and $b$ in $\mathcal{A}$.

 Let $\mathcal{A}$ be
a cochain DG algebra.  We denote by $\mathcal{A}^i$ its $i$-th homogeneous component.  The differential $\partial_{\mathcal{A}}$ is a sequence of linear maps $\partial_{\mathcal{A}}^i: \mathcal{A}^i\to \mathcal{A}^{i+1}$ such that $\partial_{\mathcal{A}}^{i+1}\circ \partial_{\mathcal{A}}^i=0$, for all $i\in \Bbb{Z}$.  If $\partial_{\mathcal{A}}\neq 0$, $\mathcal{A}$ is called
non-trivial. The cohomology graded algebra of $\mathcal{A}$ is the graded algebra $$H(\mathcal{A})=\bigoplus_{i\in \Bbb{Z}}\frac{\mathrm{ker}(\partial_{\mathcal{A}}^i)}{\mathrm{im}(\partial_{\mathcal{A}}^{i-1})}.$$
 For any cocycle element $z\in \mathrm{ker}(\partial_{\mathcal{A}}^i)$, we write $\lceil z \rceil$ as the cohomology class in $H(\mathcal{A})$ represented by $z$. One sees that $H(\mathcal{A})$ is a connected graded algebra if $\mathcal{A}$ is a connected cochain DG algebra.
  The derived category of left DG modules over $\mathcal{A}$ (DG $\mathcal{A}$-modules for short) is denoted by $\mathscr{D}(\mathcal{A})$.  A DG $\mathcal{A}$-module  $M$ is compact if the functor $\Hom_{\mathscr{D}(A)}(M,-)$ preserves
all coproducts in $\mathscr{D}(\mathcal{A})$.
 By \cite[Proposition 3.3]{MW1},
a DG $\mathcal{A}$-module  is compact if and only if it admits a minimal semi-free resolution with a finite semi-basis. The full subcategory of $\mathscr{D}(\mathcal{A})$ consisting of compact DG $\mathcal{A}$-modules is denoted by $\mathscr{D}^c(\mathcal{A})$.

  A cochain algebra $\mathcal{A}$ is called connected if its underlying graded algebra $\mathcal{A}^{\#}$ is a connected graded algebra. For any connected DG algebra $\mathcal{A}$,
   we write $\frak{m}$ as the maximal DG ideal $\mathcal{A}^{>0}$ of $\mathcal{A}$.
Via the canonical surjection $\varepsilon: \mathcal{A}\to k$, $k$ is both a DG
$\mathcal{A}$-module and a DG $\mathcal{A}\!^{op}$-module. It is easy to check that the enveloping DG algebra $\mathcal{A}^e = \mathcal{A}\otimes \mathcal{A}\!^{op}$ of $\mathcal{A}$
is also a connected cochain DG algebra with $H(\mathcal{A}^e)\cong H(\mathcal{A})^e$, and $$\frak{m}_{\mathcal{A}^e}=\frak{m}_{\mathcal{A}}\otimes \mathcal{A}^{op} + \mathcal{A}\otimes
\frak{m}_{\mathcal{A}^{op}}.$$ We have the following list of homological properties for DG algebras.
\begin{defn}\label{basicdef}
{\rm Let $\mathcal{A}$ be a connected cochain DG algebra.
\begin{enumerate}
\item  If $\dim_{k}H(R\Hom_{\mathcal{A}}(k,\mathcal{A}))=1$ (resp.$\dim_{k}H(R\Hom_{\mathcal{A}^{op}}(k,\mathcal{A}))=1$), then $\mathcal{A}$ is called left (resp. right) Gorenstein (cf. \cite{FHT1});
\item  If ${}_{\mathcal{A}}k$, or equivalently ${}_{\mathcal{A}^e}\mathcal{A}$, has a minimal semi-free resolution with a semi-basis concentrated in degree $0$, then $\mathcal{A}$ is called Koszul (cf. \cite{HW});
\item If ${}_{\mathcal{A}}k$, or equivalently the DG $\mathcal{A}^e$-module $\mathcal{A}$ is compact, then $\mathcal{A}$ is called homologically smooth (cf. \cite[Corollary 2.7]{MW3});
\item If $\mathcal{A}$ is homologically smooth and $$R\Hom_{\mathcal{A}^e}(\mathcal{A}, \mathcal{A}^e)\cong
\Sigma^{-n}\mathcal{A}$$ in  the derived category $\mathscr{D}((\mathcal{A}^e)^{op})$ of right DG $\mathcal{A}^e$-modules, then $\mathcal{A}$ is called an $n$-Calabi-Yau DG algebra  (cf. \cite{Gin,VdB}).
\end{enumerate}}
 \end{defn}
 The motivation of this paper is to study when a $3$-dimensional DG Sklyanian algebra has these homological properties in Definition \ref{basicdef}.

\section{differential structures}
In this section, we determine all possible differential structures of a $3$-dimensional DG Sklyanin algebra $\mathcal{A}$.
 Let $\mathfrak{D}$ be the subset of the projective plane $\Bbb{P}_k^2$ consisting of the $12$ points:
$$\mathfrak{D}:=\{(1,0,0), (0,1,0),(0,0,1)\}\sqcup\{(a,b,c)|a^3=b^3=c^3\}.$$
Then there exists some $(a,b,c)\in \Bbb{P}_k^2-\mathfrak{D}$ such that $\mathcal{A}^{\#}=S_{a,b,c}$. We have the following proposition on the differential $\partial_{\mathcal{A}}$ of $\mathcal{A}$.
 \begin{thm}\label{diffstr}
Let $\mathcal{A}$ be a $3$-dimensional DG Sklyanin algebra with $\mathcal{A}^{\#}=S_{a,b,c}$, $(a,b,c)\in \Bbb{P}_k^2-\mathfrak{D}$.
Then we have the following statements:

(1)$\partial_{\mathcal{A}}=0$ if either $|a|\neq |b|$ or $c\neq 0$.

(2)$\partial_{\mathcal{A}}$ is defined by
\begin{align*}
\begin{cases}
\partial_{\mathcal{A}}(x)=\alpha x^2+\beta xy +\gamma xz\\
\partial_{\mathcal{A}}(y)=\alpha yx +\beta y^2+\gamma yz\\
\partial_{\mathcal{A}}(z)=\alpha xz +\beta yz +\gamma z^2,\, \text{for some}\,(\alpha,\beta,\gamma)\in \Bbb{A}_k^3, \,\text{if}\,a=-b, c=0.
\end{cases}
\end{align*}

(3)$\partial_{\mathcal{A}}$ is defined by
\begin{align*}
\left(
                         \begin{array}{c}
                           \partial_{\mathcal{A}}(x)\\
                           \partial_{\mathcal{A}}(y)\\
                           \partial_{\mathcal{A}}(z)
                         \end{array}
                       \right)=M\left(
                         \begin{array}{c}
                           x^2\\
                           y^2\\
                           z^2
                         \end{array}
                       \right), \text{for some}\, M\in M_3(k), \,\text{if}\,a=b, c=0.
\end{align*}
\end{thm}
 \begin{proof}
 Note that
$S_{a,b,c}=\frac{k\langle x,y,z\rangle}{(f_1,f_2,f_3)},$
where $\begin{cases}
 f_1=ayz+bzy+cx^2\\
 f_2=azx+bxz+cy^2\\
 f_3=axy+byx+cz^2.
 \end{cases}$
Since $\partial_{\mathcal{A}}$  is a $k$-linear map of degree $1$, we may let
\begin{align*}
\begin{cases}
\partial_{\mathcal{A}}(x) = (x,y,z)M^x\left(
                         \begin{array}{c}
                           x\\
                           y \\
                           z
                         \end{array}
                       \right), \\
\partial_{\mathcal{A}}(y) = (x,y,z)M^y\left(
                         \begin{array}{c}
                      x \\
                      y\\
                      z
                         \end{array}
                       \right),\\
\partial_{\mathcal{A}}(z) = (x,y,z)M^z\left(
                         \begin{array}{c}
                           x\\
                           y \\
                           z
                         \end{array}
                       \right),
\end{cases}
\end{align*}

where \begin{align*}
& M^x=(c^x_1,c^x_2,c^x_3)=(m^x_{ij})_{3\times 3}=\left(
                         \begin{array}{c}
                           r_1^x\\
                           r_2^x \\
                           r_3^x
                         \end{array}
                       \right), \\
&M^y=(c^y_1,c^y_2,c^y_3)=(m^y_{ij})_{3\times 3}=\left(
                         \begin{array}{c}
                           r_1^y\\
                           r_2^y \\
                           r_3^y
                         \end{array}
                       \right) \\
\text{and}\quad
& M^z=(c^z_1,c^z_2,c^z_3)=(m^z_{ij})_{3\times 3}=\left(
                         \begin{array}{c}
                           r_1^z\\
                           r_2^z \\
                           r_3^z
                         \end{array}
                       \right) \\
 \end{align*}
  are three $3\times 3$ matrixes.
 In $\mathcal{A}^3$, we have the following system of equations
\begin{align*}
\begin{cases}\partial_{\mathcal{A}}(f_1)=0
 & \mathrm{Eq}(1)\\
\partial_{\mathcal{A}}(f_2)=0 & \mathrm{Eq}(2)\\
\partial_{\mathcal{A}}(f_3)=0 & \mathrm{Eq}(3)\\
\partial_{\mathcal{A}}^2(x)=0& \mathrm{Eq}(4)\\
\partial_{\mathcal{A}}^2(y)=0& \mathrm{Eq}(5)\\
 \partial_{\mathcal{A}}^2(z)=0 & \mathrm{Eq}(6)
\end{cases}.
\end{align*}
The equation $\mathrm{Eq}(1)$ is
\begin{align*}
0&=\partial_{\mathcal{A}}[(x,y,z)\left(
                                                                                        \begin{array}{ccc}
                                                                                          c & 0 & 0 \\
                                                                                          0 & 0 & a \\
                                                                                          0 & b & 0 \\
                                                                                        \end{array}
                                                                                      \right)
\left(
                         \begin{array}{c}
                           x \\
                           y \\
                           z
                         \end{array}
                       \right)]\\
=&(\partial_{\mathcal{A}}(x),\partial_{\mathcal{A}}(y),\partial_{\mathcal{A}}(z))\left(
                                                                                        \begin{array}{ccc}
                                                                                          c & 0 & 0 \\
                                                                                          0 & 0 & a \\
                                                                                          0 & b & 0 \\
                                                                                        \end{array}
                                                                                      \right)
\left(
                         \begin{array}{c}
                           x \\
                           y \\
                           z
                         \end{array}
                       \right)-(x,y,z)\left(
                                                                                        \begin{array}{ccc}
                                                                                          c & 0 & 0 \\
                                                                                          0 & 0 & a \\
                                                                                          0 & b & 0 \\
                                                                                        \end{array}
                                                                                      \right)\left(
                         \begin{array}{c}
                           \partial_{\mathcal{A}}(x) \\
                           \partial_{\mathcal{A}}(y) \\
                           \partial_{\mathcal{A}}(z)
                         \end{array}
                       \right)\\
=&(x,y,z) [M^x\left(
                         \begin{array}{c}
                           x \\
                           y \\
                           z
                         \end{array}
                       \right),M^y\left(
                         \begin{array}{c}
                           x \\
                           y  \\
                           z
                         \end{array}
                       \right),M^z\left(
                         \begin{array}{c}
                           x \\
                           y \\
                           z
                         \end{array}
                       \right)]\left(
                                                                                        \begin{array}{ccc}
                                                                                          c & 0 & 0 \\
                                                                                          0 & 0 & a \\
                                                                                          0 & b & 0 \\
                                                                                        \end{array}
                                                                                      \right)\left(
                         \begin{array}{c}
                           x\\
                           y \\
                           z
                         \end{array}
                       \right)\\
& - (x,y,z)\left(
                                                                                        \begin{array}{ccc}
                                                                                          c & 0 & 0 \\
                                                                                          0 & 0 & a \\
                                                                                          0 & b & 0 \\
                                                                                        \end{array}
                                                                                      \right)\left(
                         \begin{array}{c}
                           (x,y,z)M^x \\
                           (x,y,z)M^y  \\
                           (x,y,z)M^z
                         \end{array}
                       \right)\left(
                         \begin{array}{c}
                           x \\
                           y  \\
                           z
                         \end{array}
                       \right)\\
=&(x,y,z) (xc^x_1+yc^x_2+zc^x_3,xc^y_1+yc^y_2+zc^y_3,xc^z_1+yc^z_2+zc^z_3)\left(
                                                                                        \begin{array}{ccc}
                                                                                          c & 0 & 0 \\
                                                                                          0 & 0 & a \\
                                                                                          0 & b & 0 \\
                                                                                        \end{array}
                                                                                      \right)\left(
                         \begin{array}{c}
                           x\\
                           z \\
                           y
                         \end{array}
                       \right)\\
& - (x,y,z)\left(
                                                                                        \begin{array}{ccc}
                                                                                          c & 0 & 0 \\
                                                                                          0 & 0 & a \\
                                                                                          0 & b & 0 \\
                                                                                        \end{array}
                                                                                      \right)\left(
                         \begin{array}{c}
                           xr_1^x+yr_2^x+zr_3^x \\
                           xr_1^y+yr_2^y+zr_3^y  \\
                           xr_1^z+yr_2^z+zr_3^z
                         \end{array}
                       \right)\left(
                         \begin{array}{c}
                           x \\
                           y  \\
                           z
                         \end{array}
                       \right).
     \end{align*}
Similarly, $\mathrm{Eq}(2)$ and $\mathrm{Eq}(3)$ are
\begin{align*}
0=&(x,y,z) (xc^x_1+yc^x_2+zc^x_3,xc^y_1+yc^y_2+zc^y_3,xc^z_1+yc^z_2+zc^z_3)\left(
                                                                                        \begin{array}{ccc}
                                                                                          0 & 0 & b \\
                                                                                          0 & c & 0 \\
                                                                                          a & 0 & 0 \\
                                                                                        \end{array}
                                                                                      \right)\left(
                         \begin{array}{c}
                           x\\
                           z \\
                           y
                         \end{array}
                       \right)\\
& - (x,y,z)\left(
                                                                                        \begin{array}{ccc}
                                                                                          0 & 0 & b \\
                                                                                          0 & c & 0 \\
                                                                                          a & 0 & 0 \\
                                                                                        \end{array}
                                                                                      \right)\left(
                         \begin{array}{c}
                           xr_1^x+yr_2^x+zr_3^x \\
                           xr_1^y+yr_2^y+zr_3^y  \\
                           xr_1^z+yr_2^z+zr_3^z
                         \end{array}
                       \right)\left(
                         \begin{array}{c}
                           x \\
                           y  \\
                           z
                         \end{array}
                       \right)
\end{align*}
and
\begin{align*}
0=&(x,y,z) (xc^x_1+yc^x_2+zc^x_3,xc^y_1+yc^y_2+zc^y_3,xc^z_1+yc^z_2+zc^z_3)\left(
                                                                                        \begin{array}{ccc}
                                                                                          0 & a & 0 \\
                                                                                          b & 0 & 0 \\
                                                                                          0 & 0 & c \\
                                                                                        \end{array}
                                                                                      \right)\left(
                         \begin{array}{c}
                           x\\
                           z \\
                           y
                         \end{array}
                       \right)\\
& - (x,y,z)\left(
                                                                                        \begin{array}{ccc}
                                                                                          0 & a & 0 \\
                                                                                          b & 0 & 0 \\
                                                                                          0 & 0 & c \\
                                                                                        \end{array}
                                                                                      \right)\left(
                         \begin{array}{c}
                           xr_1^x+yr_2^x+zr_3^x \\
                           xr_1^y+yr_2^y+zr_3^y  \\
                           xr_1^z+yr_2^z+zr_3^z
                         \end{array}
                       \right)\left(
                         \begin{array}{c}
                           x \\
                           y  \\
                           z
                         \end{array}
                       \right),
\end{align*}
respectively.
For $\mathrm{Eq}(4)$,$\mathrm{Eq}(5)$ and $\mathrm{Eq}(6)$, we can also expand them similarly. For example, in $\mathcal{A}^3$, $\mathrm{Eq}(4)$ is
\begin{align*}
0&=\partial_{\mathcal{A}}\circ\partial_{\mathcal{A}}(x)=\partial_{\mathcal{A}}[(x,y,z)M^x\left(
                         \begin{array}{c}
                           x \\
                           y \\
                           z
                         \end{array}
                       \right)]\\
                    & =(\partial_{\mathcal{A}}(x),\partial_{\mathcal{A}}(y),\partial_{\mathcal{A}}(z)) M^x\left(
                         \begin{array}{c}
                           x \\
                           y \\
                           z
                         \end{array}
                       \right)  - (x,y,z)M^x \left(
                         \begin{array}{c}
                           \partial_{\mathcal{A}}(x) \\
                           \partial_{\mathcal{A}}(y)  \\
                           \partial_{\mathcal{A}}(z)
                         \end{array}
                       \right)\\
                   &   =(x,y,z) [M^x\left(
                         \begin{array}{c}
                           x \\
                           y \\
                           z
                         \end{array}
                       \right),M^y\left(
                         \begin{array}{c}
                           x \\
                           y  \\
                           z
                         \end{array}
                       \right),M^z\left(
                         \begin{array}{c}
                           x \\
                           y \\
                           z
                         \end{array}
                       \right)]M^x\left(
                         \begin{array}{c}
                           x\\
                           y \\
                           z
                         \end{array}
                       \right)\\
& - (x,y,z)M^x\left(
                         \begin{array}{c}
                           (x,y,z)M^x \\
                           (x,y,z)M^y  \\
                           (x,y,z)M^z
                         \end{array}
                       \right)\left(
                         \begin{array}{c}
                           x \\
                           y  \\
                           z
                         \end{array}
                       \right)\\
=&(x,y,z)[(xc_1^x+yc_2^x+zc_3^x)r^x_1+ (xc_1^y+yc_2^y+zc_3^y)r^x_2+(xc_1^z+yc_2^z+zc_3^z)r^x_3]\left(
                         \begin{array}{c}
                           x \\
                           y \\
                           z
                         \end{array}
                       \right)\\
-&(x,y,z)[c^x_1(xr_1^x+yr_2^x+zr_3^x)+c^x_2( xr_1^y+yr_2^y+zr_3^y )+c_3^x(xr_1^z+yr_2^z+zr_3^z)]
                        \left(
                         \begin{array}{c}
                           x\\
                           y \\
                           z
                         \end{array}
                       \right)\\
=&(x,y,z)[(yc_2^x+zc_3^x)r^x_1+ (xc_1^y+yc_2^y+zc_3^y)r^x_2+(xc_1^z+yc_2^z+zc_3^z)r^x_3]\left(
                         \begin{array}{c}
                           x \\
                           y \\
                           z
                         \end{array}
                       \right)\\
-&(x,y,z)[c^x_1(yr_2^x+zr_3^x)+c^x_2( xr_1^y+yr_2^y+zr_3^y )+c_3^x(xr_1^z+yr_2^z+zr_3^z)]
                        \left(
                         \begin{array}{c}
                           x\\
                           y \\
                           z
                         \end{array}
                       \right).
\end{align*}
By similar computations, $\mathrm{Eq}(5)$ and $\mathrm{Eq}(6)$ are
\begin{align*}
0&=(x,y,z)[(xc_1^x+yc_2^x+zc_3^x)r_1^y+(xc_1^y+zc_3^y)r_2^y+(xc_1^z+yc_2^z+zc_3^z)r_3^y] \left(
                         \begin{array}{c}
                           x\\
                           y \\
                           z
                         \end{array}
                       \right) \\
 &-(x,y,z)[c_1^y(xr_1^x+yr_2^x+zr_3^x)+c_2^y(xr_1^y+zr_3^y)+c_3^y(xr_1^z+yr_2^z+zr_3^z)]\left(
                         \begin{array}{c}
                           x\\
                           y \\
                           z
                         \end{array}
                       \right)
\end{align*}
and
\begin{align*}
  0&=(x,y,z)[(xc_1^x+yc_2^x+zc_3^x)r_1^z+(xc_1^y+yc_2^y+zc_3^y)r_2^z+(xc_1^z+yc_2^z)r_3^z] \left(
                         \begin{array}{c}
                           x\\
                           y \\
                           z
                         \end{array}
                       \right) \\
 &-(x,y,z)[c_1^z(xr_1^x+yr_2^x+zr_3^x)+c_2^z(xr_1^y+yr_2^y+zr_3^y)+c_3^z(xr_1^z+yr_2^z)]\left(
                         \begin{array}{c}
                           x\\
                           y \\
                           z
                         \end{array}
                       \right),
\end{align*}
respectively. In order to study the solutions of $\mathrm{Eq}(1)\sim \mathrm{Eq}(6)$, we divide all
$3$-dimensional DG Sklyanin algebras into the following $4$ case:
\begin{align*}& \mathrm{Case}\,1.\quad a=0,b\neq 0,c\neq 0; \\
&\mathrm{Case}\,2.\quad  b=0,a\neq 0,c\neq 0; \\
&\mathrm{Case}\,3. \quad a\neq 0,b\neq 0,c\neq 0;\\
& \mathrm{Case}\,4. \quad  c=0,a\neq 0,b\neq 0.
\end{align*}

 In Case $1$, we have $b=0,a\in k^{\times} ,c\in k^{\times}$. One sees that $S_{a,b,c}^3$ has a basis
$$\{x^3,x^2y,x^2z,xyx,xzx,xyz,yx^2,yxy,yzx,zxy \}.$$  Via some routine and tedious computations of $\mathrm{Eq}(1), \mathrm{Eq}(2)$ and $\mathrm{Eq}(3)$, we can see that they are equivalent to
\begin{align}\label{f1}
\begin{cases}
m_{23}^x=0\\
m_{31}^y=0\\
m_{12}^z=0\\
m^x_{12}=m^z_{23}\\
m^x_{31}=m^y_{23}\\
m_{12}^y=m^z_{31}\\
m_{22}^x=\frac{c}{b}(m_{13}^x-m_{31}^x)\\
m_{33}^x=\frac{c}{b}(m_{21}^x-m_{12}^x)\\
m_{11}^y=\frac{c}{b}(m_{32}^y-m_{23}^y)\\
m_{33}^y=\frac{c}{b}(m_{21}^y-m_{12}^y)\\
m_{11}^z=\frac{c}{b}(m_{32}^z-m_{23}^z)\\
m_{22}^z=\frac{c}{b}(m_{13}^z-m_{31}^z)\\
m_{11}^x=2m_{31}^z+\frac{c}{b}m_{32}^x \\
m_{22}^y=2m_{12}^x+\frac{c}{b}m_{13}^y\\
m_{33}^z=2m_{23}^y+\frac{c}{b}m_{21}^z.
\end{cases}
\end{align}
Substituting (\ref{f1}) into the $30$ equations obtained by $\mathrm{Eq}(4)$, $\mathrm{Eq}(5)$ and $\mathrm{Eq}(6)$, we see that those equations are equivalent to $m_{12}^x=m_{31}^x=m_{31}^z=0.$
Therefore, the equations $\mathrm{Eq}(1)\sim \mathrm{Eq}(6)$ are equivalent to
\begin{align*}
\begin{cases}
m^x_{12}=m_{23}^x=m^x_{31}=m^y_{12}=m^y_{23}=m_{31}^y=m_{12}^z=m^z_{23}=m_{31}^z=0\\
m_{11}^x=\frac{c}{b}m_{32}^x \\
m_{22}^x=\frac{c}{b}m_{13}^x\\
m_{33}^x=\frac{c}{b}m_{21}^x\\
m_{11}^y=\frac{c}{b}m_{32}^y\\
m_{22}^y=\frac{c}{b}m_{13}^y\\
m_{33}^y=\frac{c}{b}m_{21}^y\\
m_{11}^z=\frac{c}{b}m_{32}^z\\
m_{22}^z=\frac{c}{b}m_{13}^z\\
m_{33}^z=\frac{c}{b}m_{21}^z.
\end{cases}
\end{align*}
Then  $\partial_{\mathcal{A}}$ is defined  by
$$\begin{cases}
\partial_{\mathcal{A}}(x) = (x,y,z)\left(
                                      \begin{array}{ccc}
                                        \alpha_1 & 0 & \frac{b}{c}\alpha_2 \\
                                        \frac{b}{c}\alpha_3 & \alpha_2 & 0 \\
                                        0 & \frac{b}{c}\alpha_1 & \alpha_3 \\
                                      \end{array}
                                    \right)
\left(
                         \begin{array}{c}
                           x\\
                           y \\
                           z
                         \end{array}
                       \right)\\
\partial_{\mathcal{A}}(y) = (x,y,z)\left(
                                      \begin{array}{ccc}
                                        \beta_1 & 0 & \frac{b}{c}\beta_2 \\
                                        \frac{b}{c}\beta_3 & \beta_2 & 0 \\
                                        0 & \frac{b}{c}\beta_1 & \beta_3 \\
                                      \end{array}
                                    \right)\left(
                         \begin{array}{c}
                      x \\
                      y\\
                      z
                         \end{array}
                       \right) \\
\partial_{\mathcal{A}}(z) = (x,y,z)\left(
                                      \begin{array}{ccc}
                                        \gamma_1 & 0 & \frac{b}{c}\gamma_2 \\
                                        \frac{b}{c}\gamma_3 & \gamma_2 & 0 \\
                                        0 & \frac{b}{c}\gamma_1 & \gamma_3 \\
                                      \end{array}
                                    \right)\left(
                         \begin{array}{c}
                           x\\
                           y \\
                           z
                         \end{array}
                       \right),
\end{cases}$$
for some $(\alpha_1,\alpha_2,\alpha_3,\beta_1,\beta_2,\beta_3, \gamma_1, \gamma_2,\gamma_3)\in \Bbb{A}_k^9 $.
 Since
 \begin{align*}
 \begin{cases}
 bzy+cx^2=0\\
  bxz+cy^2=0\\
  byx+cz^2=0
  \end{cases}
 \end{align*}
  in $\mathcal{A}^2$, we have \begin{align*}
                     \partial_{\mathcal{A}}(x)&=(x,y,z)\left(
                                      \begin{array}{ccc}
                                        \alpha_1 & 0 & \frac{b}{c}\alpha_2 \\
                                        \frac{b}{c}\alpha_3 & \alpha_2 & 0 \\
                                        0 & \frac{b}{c}\alpha_1 & \alpha_3 \\
                                      \end{array}
                                    \right)
\left(
                         \begin{array}{c}
                           x\\
                           y \\
                           z
                         \end{array}
                       \right)\\
                       &=\alpha_1x^2+\frac{b}{c}\alpha_2xz+\frac{b}{c}\alpha_3yx+\alpha_2y^2+\frac{b}{c}\alpha_1zy+\alpha_3z^2\\
                       &=\frac{\alpha_1}{c}(cx^2+bzy) + \frac{\alpha_2}{c}(cy^2+bxz) +\frac{\alpha_3}{c}(cz^2+byx) \\
                       &=0.
\end{align*}
Similarly, we can show that $\partial_{\mathcal{A}}(y)=\partial_{\mathcal{A}}(z)=0$. Hence $\partial_{\mathcal{A}}=0$.

In Case $2$, we have $b=0, a,c\in k^{\times}$. One sees that $S_{a,b,c}^3$ admits a $k$-linear basis
$$\{x^3,x^2y,x^2z,xy^2,xyx,xzy,yx^2,yxz,y^2x,zyx \}.$$ By computations of $\mathrm{Eq}(1), \mathrm{Eq}(2)$ and $\mathrm{Eq}(3)$, we can see that they are equivalent to
\begin{align}\label{f2}
\begin{cases}
m_{32}^x=0\\
m_{13}^y=0\\
m_{21}^z=0\\
m^x_{13}=m^y_{32}\\
m^x_{21}=m^z_{32}\\
m_{21}^y=m^z_{13}\\
m_{22}^x=\frac{c}{a}(m_{31}^x-m_{13}^x)\\
m_{33}^x=\frac{c}{a}(m_{12}^x-m_{21}^x)\\
m_{11}^y=\frac{c}{a}(m_{23}^y-m_{32}^y)\\
m_{33}^y=\frac{c}{a}(m_{12}^y-m_{21}^y)\\
m_{11}^z=\frac{c}{a}(m_{23}^z-m_{32}^z)\\
m_{22}^z=\frac{c}{a}(m_{31}^z-m_{13}^z)\\
m_{11}^x=2m_{21}^y+\frac{c}{a}m_{23}^x \\
m_{22}^y=2m_{21}^x+\frac{c}{a}m_{31}^y\\
m_{33}^z=2m_{13}^x+\frac{c}{a}m_{12}^z.
\end{cases}
\end{align}
Substituting (\ref{f2}) into the $30$ equations obtained by $\mathrm{Eq}(4)$, $\mathrm{Eq}(5)$ and $\mathrm{Eq}(6)$, we see that those equations are equivalent to
$m_{13}^x=m_{21}^x=m_{13}^z=0.$
Therefore, the equations $\mathrm{Eq}(1)\sim \mathrm{Eq}(6)$ are equivalent to
\begin{align*}
\begin{cases}
m^x_{13}=m_{21}^x=m^x_{32}=m^y_{13}=m^y_{21}=m_{32}^y=m_{13}^z=m^z_{21}=m_{32}^z=0\\
m_{11}^x=\frac{c}{a}m_{23}^x \\
m_{22}^x=\frac{c}{a}m_{31}^x\\
m_{33}^x=\frac{c}{a}m_{12}^x\\
m_{11}^y=\frac{c}{a}m_{23}^y\\
m_{22}^y=\frac{c}{a}m_{31}^y\\
m_{33}^y=\frac{c}{a}m_{12}^y\\
m_{11}^z=\frac{c}{a}m_{23}^z\\
m_{22}^z=\frac{c}{a}m_{31}^z\\
m_{33}^z=\frac{c}{a}m_{12}^z.
\end{cases}
\end{align*}
Then  $\partial_{\mathcal{A}}$ is defined  by
$$\begin{cases}
\partial_{\mathcal{A}}(x) = (x,y,z)\left(
                                      \begin{array}{ccc}
                                        \alpha_1 & \frac{a}{c}\alpha_3 &0\\
                                        0 & \alpha_2 & \frac{a}{c}\alpha_1 \\
                                         \frac{a}{c}\alpha_2  & 0 & \alpha_3 \\
                                      \end{array}
                                    \right)
\left(
                         \begin{array}{c}
                           x\\
                           y \\
                           z
                         \end{array}
                       \right) \\
\partial_{\mathcal{A}}(y) = (x,y,z)\left(
                                      \begin{array}{ccc}
                                        \beta_1 & \frac{a}{c}\beta_3 &0\\
                                        0 & \beta_2 & \frac{a}{c}\beta_1 \\
                                         \frac{a}{c}\beta_2  & 0 & \beta_3 \\
                                      \end{array}
                                    \right)\left(
                         \begin{array}{c}
                      x \\
                      y\\
                      z
                         \end{array}
                       \right) \\
\partial_{\mathcal{A}}(z) = (x,y,z)\left(
                                      \begin{array}{ccc}
                                        \gamma_1 & \frac{a}{c}\gamma_3 &0\\
                                        0 & \gamma_2 & \frac{a}{c}\gamma_1 \\
                                         \frac{a}{c}\gamma_2  & 0 & \gamma_3 \\
                                      \end{array}
                                    \right)\left(
                         \begin{array}{c}
                           x\\
                           y \\
                           z
                         \end{array}
                       \right),
                       \end{cases}$$
for some $(\alpha_1,\alpha_2,\alpha_3,\beta_1,\beta_2,\beta_3, \gamma_1, \gamma_2,\gamma_3)\in \Bbb{A}_k^9 $. Since
$$\begin{cases}
ayz+cx^2=0\\
azx+cy^2=0\\
axy+cz^2=0
\end{cases}$$
in $\mathcal{A}^2$,  we have \begin{align*}
                     \partial_{\mathcal{A}}(x)&=(x,y,z)\left(
                                      \begin{array}{ccc}
                                        \alpha_1 & \frac{a}{c}\alpha_3 &0\\
                                        0 & \alpha_2 & \frac{a}{c}\alpha_1 \\
                                         \frac{a}{c}\alpha_2  & 0 & \alpha_3 \\
                                      \end{array}
                                    \right)
\left(
                         \begin{array}{c}
                           x\\
                           y \\
                           z
                         \end{array}
                       \right)\\
                       &=\alpha_1x^2+\frac{a}{c}\alpha_3xy+\frac{a}{c}\alpha_1yz+\alpha_2y^2+\frac{a}{c}\alpha_2zx+\alpha_3z^2\\
                       &=\frac{\alpha_1}{c}(cx^2+ayz)+\frac{\alpha_2}{c}(cy^2+azx)+\frac{\alpha_3}{c}(cz^2+axy)\\
                       &=0.
\end{align*}
Similarly, we can show that $\partial_{\mathcal{A}}(y)=\partial_{\mathcal{A}}(z)=0$. Hence $\partial_{\mathcal{A}}=0$.

In Case $3$, we have $ a,b, c\in k^{\times}$.  One sees that $\mathcal{A}^3=S_{a,b,c}^3$ has a $k$-linear basis
$$\{x^2y,x^2z,xy^2,xz^2,yx^2,y^2x,y^2z,yz^2,xyz,xzy,yzx,yxz\}.$$
By computations, one sees that $\mathrm{Eq}(1), \mathrm{Eq}(2)$ and $\mathrm{Eq}(3)$ are equivalent to
\begin{align*}
\begin{cases}
(c-\frac{b^2c}{a^2})m^z_{11}-(a+\frac{b^2}{a})m^y_{13}+\frac{c^2}{b}m^z_{23}+(b+\frac{b^3}{a^2})m^y_{31}-\frac{c^2}{b}m^z_{32}=0\\
-\frac{bc^2}{a^2}m^z_{11}+am^x_{12}-\frac{bc}{a}m^y_{13}-am^y_{22}+(c+\frac{b^2c}{a^2})m^y_{31}+am^z_{32}-\frac{a^2}{c}m^x_{33}=0\\
\frac{c^2}{a}m^z_{11}+\frac{b^2}{a}m^x_{12}-\frac{bc}{a}m^y_{13}-am^y_{22}+(c+\frac{b^2c}{a^2})m^y_{31}+am^z_{32}-\frac{a^2}{c}m^x_{33}=0\\
\frac{bc}{a}m^x_{12}-cm^x_{21}+(c+\frac{ac}{b})m^z_{23}-(\frac{ac}{b}+\frac{bc}{a})m^z_{32}=0\\
\frac{b^2-ab}{c}m^z_{11}+bm^x_{12}+cm^y_{13}-am^x_{21}-\frac{bc}{a}m^y_{31}+(a-b)m^z_{32}=0\\
\frac{ab-a^2}{c}m^z_{11}+bm^x_{12}-\frac{a^2}{b}m^x_{21}-bm^z_{23}+\frac{a^2}{b}m^z_{32}+\frac{a^2-ab}{c}m^x_{33}=0\\
-am^x_{12}+\frac{a^2}{b}m^x_{21}+am^z_{23}-\frac{a^2}{b}m^z_{32}=0\\
-bm^x_{12}-cm^y_{13}+bm^x_{21}+\frac{bc}{a}m^y_{31}+\frac{ab-b^2}{c}m^x_{33}=0\\
bm^z_{11}-\frac{ac+bc}{b}m^x_{12}+\frac{ac}{b}m^x_{21}+cm^y_{22}-\frac{a^2}{c}m^y_{31}-cm^z_{32}+\frac{a^2}{b}m^x_{33}=0\\
-\frac{c^2}{b}m^x_{12}+\frac{a^2+b^2}{a}m^y_{13}+\frac{c^2}{b}m^x_{21}-(a+b)m^y_{31}+\frac{ac-bc}{b}m^x_{33}=0\\
-am^z_{11}+cm^x_{21}-\frac{ac}{b}m^y_{22}-cm^z_{23}+\frac{b^2}{c}m^y_{31}+\frac{2ac}{b}m^z_{32}-bm^x_{33}=0\\
\frac{bc}{a}m^y_{13}+\frac{ab-a^2}{b}m^y_{22}-2am^z_{23}-cm^y_{31}+\frac{2a^2}{b}m^z_{32}=0\\
-\frac{b^2c}{a^2}m^y_{11}+bm^z_{12}-\frac{c^2}{a}m^x_{13}-cm^x_{22}+\frac{bc^2+ac^2}{a^2}m^x_{31}+\frac{ab}{c}m^y_{32}-\frac{ab}{c}m^z_{33}=0\\
\frac{a^3-b^3}{a^2}m^y_{11}-\frac{ac+bc}{a}m^x_{13}+cm^y_{23}+\frac{b^2c+abc}{a^2}m^x_{31}-\frac{ac}{b}m^y_{32}=0\\
\frac{bc}{a}m^y_{11}+bm^z_{12}-(a+\frac{b^2}{a})m^z_{21}+cm^x_{22}-\frac{c^2}{a}m^x_{31}-\frac{b^3}{ac}m^y_{32}+\frac{b^3}{ac}m^z_{33}=0\\
cm^z_{12}-\frac{bc}{a}m^z_{21}+2am^y_{23}-(\frac{b^2}{a}+\frac{a^2}{b})m^y_{32}+(\frac{b^2}{a}-a)m^z_{33}=0\\
cm^y_{11}-(b+\frac{a^2}{b})m^z_{12}+am^z_{21}+\frac{ac}{b}m^x_{22}-\frac{a^3}{bc}m^x_{31}-\frac{c^2}{b}m^y_{32}+\frac{a^3}{bc}m^z_{33}=0\\
-\frac{ac}{b}m^z_{12}+2bm^x_{13}+cm^z_{21}-bm^z_{23}-(\frac{b^2}{a}+\frac{a^2}{b})m^x_{31}+(\frac{a^2}{b}-b)m^z_{33}=0\\
-cm^y_{11}+am^z_{21}-\frac{a^2c}{b^2}m^x_{22}-\frac{c^2}{b}m^y_{23}+\frac{ab}{c}m^x_{31}+(\frac{c^2}{b}+\frac{ac^2}{b^2})m^y_{32}-\frac{ab}{c}m^z_{33}=0\\
cm^x_{13}+(b-\frac{a^3}{b^2})m^x_{22}-(c+\frac{ac}{b})m^y_{23}-\frac{bc}{a}m^x_{31}+(\frac{ac}{b}+\frac{a^2c}{b^2}m^y_{32}=0\\
\frac{ab-b^2}{c}m^y_{11}+\frac{bc}{a}m^z_{12}+am^x_{13}-cm^z_{21}-bm^x_{31}+(b-a)m^y_{32}=0\\
\frac{a^2-ab}{c}m^y_{11}+cm^z_{12}-\frac{ac}{b}m^x_{22}-am^x_{31}+am^y_{32}=0\\
-cm^z_{12}-\frac{ac}{b}m^z_{21}-am^y_{23}+am^y_{32}=0\\
-\frac{bc}{a}m^z_{12}-bm^x_{13}+cm^z_{21}+\frac{b^2-ab}{c}m^x_{22}+bm^x_{31}=0\\
\frac{b^2}{a}m^y_{12}+bm^z_{13}-bm^y_{21}-\frac{b^2}{a}m^z_{31}=0\\
am^y_{12}-am^y_{21}-cm^x_{23}+\frac{ac}{b}m^x_{32}+\frac{ab-a^2}{c}m^y_{33}=0\\
-bm^y_{12}+am^y_{21}+\frac{a^2-ab}{c})m^z_{22}+cm^x_{23}+(b-a)m^z_{31}-\frac{ac}{b}m^x_{32}=0\\
-\frac{b^2}{a}m^y_{12}-am^z_{13}+am^y_{21}+\frac{ab-b^2}{c}m^z_{22}+\frac{b^2}{a}m^z_{31}+\frac{b^2-ab}{c}m^y_{33}=0\\
-\frac{bc}{a}m^x_{11}+cm^y_{12}-cm^z_{13}-bm^z_{22}+\frac{2bc}{a}m^z_{31}+\frac{a^2}{c}m^x_{32}-am^y_{33}=0\\
(b-\frac{b^2}{a})m^x_{11}-2bm^z_{13}+\frac{ac}{b}m^x_{23}+\frac{2b^2}{a}m^z_{31}-cm^x_{32}=0\\
cm^x_{11}+\frac{bc}{a}m^y_{12}-(c+\frac{bc}{a})m^y_{21}+am^z_{22}-cm^z_{31}-\frac{b^2}{c}m^x_{32}+\frac{b^2}{a}m^y_{33}=0\\
\frac{c^2}{a}m^y_{12}-\frac{c^2}{a}m^y_{21}+(b+\frac{a^2}{b})m^x_{23}-(a+b)m^x_{32}+(\frac{bc}{a}-c)m^y_{33}=0\\
am^x_{11}-2am^y_{12}+\frac{a^2}{b}m^y_{21}+\frac{c^2}{b})m^z_{22}-\frac{a^2}{b}m^z_{31}-\frac{ac}{b}m^x_{32}+\frac{a^2}{c}m^y_{33}=0\\
-cm^y_{12}+(c+\frac{bc}{a})m^z_{13}+\frac{ac}{b}m^y_{21}-(\frac{bc}{a}+\frac{ac}{b})m^z_{31}=0\\
-bm^x_{11}+bm^y_{21}-\frac{ac^2}{b^2}m^z_{22}-\frac{ac}{b}m^x_{23}+bm^z_{31}+(c+\frac{a^2c}{b^2})m^x_{32}-\frac{b^2}{c}m^y_{33}=0\\
\frac{c^2}{a}m^z_{13}+(c-\frac{a^2c}{b^2})m^z_{22}-(b+\frac{a^2}{b})m^x_{23}-\frac{c^2}{a}m^z_{31}+(a+\frac{a^3}{b^2})m^x_{32}=0.\\
\end{cases}
\end{align*}
Note that
the equations above can be divided into the following three systems of equations:
\begin{small}
\begin{align}\label{typeone}
\begin{cases}
(c-\frac{b^2c}{a^2})m^z_{11}-(a+\frac{b^2}{a})m^y_{13}+\frac{c^2}{b}m^z_{23}+(b+\frac{b^3}{a^2})m^y_{31}-\frac{c^2}{b}m^z_{32}=0\\
-\frac{bc^2}{a^2}m^z_{11}+am^x_{12}-\frac{bc}{a}m^y_{13}-am^y_{22}+(c+\frac{b^2c}{a^2})m^y_{31}+am^z_{32}-\frac{a^2}{c}m^x_{33}=0\\
\frac{c^2}{a}m^z_{11}+\frac{b^2}{a}m^x_{12}-\frac{bc}{a}m^y_{13}-am^y_{22}+(c+\frac{b^2c}{a^2})m^y_{31}+am^z_{32}-\frac{a^2}{c}m^x_{33}=0\\
\frac{bc}{a}m^x_{12}-cm^x_{21}+(c+\frac{ac}{b})m^z_{23}-(\frac{ac}{b}+\frac{bc}{a})m^z_{32}=0\\
\frac{b^2-ab}{c}m^z_{11}+bm^x_{12}+cm^y_{13}-am^x_{21}-\frac{bc}{a}m^y_{31}+(a-b)m^z_{32}=0\\
\frac{ab-a^2}{c}m^z_{11}+bm^x_{12}-\frac{a^2}{b}m^x_{21}-bm^z_{23}+\frac{a^2}{b}m^z_{32}+\frac{a^2-ab}{c}m^x_{33}=0\\
-am^x_{12}+\frac{a^2}{b}m^x_{21}+am^z_{23}-\frac{a^2}{b}m^z_{32}=0\\
-bm^x_{12}-cm^y_{13}+bm^x_{21}+\frac{bc}{a}m^y_{31}+\frac{ab-b^2}{c}m^x_{33}=0\\
bm^z_{11}-\frac{ac+bc}{b}m^x_{12}+\frac{ac}{b}m^x_{21}+cm^y_{22}-\frac{a^2}{c}m^y_{31}-cm^z_{32}+\frac{a^2}{b}m^x_{33}=0\\
-\frac{c^2}{b}m^x_{12}+\frac{a^2+b^2}{a}m^y_{13}+\frac{c^2}{b}m^x_{21}-(a+b)m^y_{31}+\frac{ac-bc}{b}m^x_{33}=0\\
-am^z_{11}+cm^x_{21}-\frac{ac}{b}m^y_{22}-cm^z_{23}+\frac{b^2}{c}m^y_{31}+\frac{2ac}{b}m^z_{32}-bm^x_{33}=0\\
\frac{bc}{a}m^y_{13}+\frac{ab-a^2}{b}m^y_{22}-2am^z_{23}-cm^y_{31}+\frac{2a^2}{b}m^z_{32}=0,
\end{cases}
\end{align}
\end{small}
\begin{small}
\begin{align}\label{typetwo}
\begin{cases}
-\frac{b^2c}{a^2}m^y_{11}+bm^z_{12}-\frac{c^2}{a}m^x_{13}-cm^x_{22}+\frac{bc^2+ac^2}{a^2}m^x_{31}+\frac{ab}{c}m^y_{32}-\frac{ab}{c}m^z_{33}=0\\
\frac{a^3-b^3}{a^2}m^y_{11}-\frac{ac+bc}{a}m^x_{13}+cm^y_{23}+\frac{b^2c+abc}{a^2}m^x_{31}-\frac{ac}{b}m^y_{32}=0\\
\frac{bc}{a}m^y_{11}+bm^z_{12}-(a+\frac{b^2}{a})m^z_{21}+cm^x_{22}-\frac{c^2}{a}m^x_{31}-\frac{b^3}{ac}m^y_{32}+\frac{b^3}{ac}m^z_{33}=0\\
cm^z_{12}-\frac{bc}{a}m^z_{21}+2am^y_{23}-(\frac{b^2}{a}+\frac{a^2}{b})m^y_{32}+(\frac{b^2}{a}-a)m^z_{33}=0\\
cm^y_{11}-(b+\frac{a^2}{b})m^z_{12}+am^z_{21}+\frac{ac}{b}m^x_{22}-\frac{a^3}{bc}m^x_{31}-\frac{c^2}{b}m^y_{32}+\frac{a^3}{bc}m^z_{33}=0\\
-\frac{ac}{b}m^z_{12}+2bm^x_{13}+cm^z_{21}-bm^z_{23}-(\frac{b^2}{a}+\frac{a^2}{b})m^x_{31}+(\frac{a^2}{b}-b)m^z_{33}=0\\
-cm^y_{11}+am^z_{21}-\frac{a^2c}{b^2}m^x_{22}-\frac{c^2}{b}m^y_{23}+\frac{ab}{c}m^x_{31}+(\frac{c^2}{b}+\frac{ac^2}{b^2})m^y_{32}-\frac{ab}{c}m^z_{33}=0\\
cm^x_{13}+(b-\frac{a^3}{b^2})m^x_{22}-(c+\frac{ac}{b})m^y_{23}-\frac{bc}{a}m^x_{31}+(\frac{ac}{b}+\frac{a^2c}{b^2}m^y_{32}=0\\
\frac{ab-b^2}{c}m^y_{11}+\frac{bc}{a}m^z_{12}+am^x_{13}-cm^z_{21}-bm^x_{31}+(b-a)m^y_{32}=0\\
\frac{a^2-ab}{c}m^y_{11}+cm^z_{12}-\frac{ac}{b}m^x_{22}-am^x_{31}+am^y_{32}=0\\
-cm^z_{12}-\frac{ac}{b}m^z_{21}-am^y_{23}+am^y_{32}=0\\
-\frac{bc}{a}m^z_{12}-bm^x_{13}+cm^z_{21}+\frac{b^2-ab}{c}m^x_{22}+bm^x_{31}=0,
\end{cases}
\end{align}
\end{small}
\begin{small}
\begin{align}\label{typethree}
\begin{cases}
\frac{b^2}{a}m^y_{12}+bm^z_{13}-bm^y_{21}-\frac{b^2}{a}m^z_{31}=0\\
am^y_{12}-am^y_{21}-cm^x_{23}+\frac{ac}{b}m^x_{32}+\frac{ab-a^2}{c}m^y_{33}=0\\
-bm^y_{12}+am^y_{21}+\frac{a^2-ab}{c})m^z_{22}+cm^x_{23}+(b-a)m^z_{31}-\frac{ac}{b}m^x_{32}=0\\
-\frac{b^2}{a}m^y_{12}-am^z_{13}+am^y_{21}+\frac{ab-b^2}{c}m^z_{22}+\frac{b^2}{a}m^z_{31}+\frac{b^2-ab}{c}m^y_{33}=0\\
-\frac{bc}{a}m^x_{11}+cm^y_{12}-cm^z_{13}-bm^z_{22}+\frac{2bc}{a}m^z_{31}+\frac{a^2}{c}m^x_{32}-am^y_{33}=0\\
(b-\frac{b^2}{a})m^x_{11}-2bm^z_{13}+\frac{ac}{b}m^x_{23}+\frac{2b^2}{a}m^z_{31}-cm^x_{32}=0\\
cm^x_{11}+\frac{bc}{a}m^y_{12}-(c+\frac{bc}{a})m^y_{21}+am^z_{22}-cm^z_{31}-\frac{b^2}{c}m^x_{32}+\frac{b^2}{a}m^y_{33}=0\\
\frac{c^2}{a}m^y_{12}-\frac{c^2}{a}m^y_{21}+(b+\frac{a^2}{b})m^x_{23}-(a+b)m^x_{32}+(\frac{bc}{a}-c)m^y_{33}=0\\
am^x_{11}-2am^y_{12}+\frac{a^2}{b}m^y_{21}+\frac{c^2}{b})m^z_{22}-\frac{a^2}{b}m^z_{31}-\frac{ac}{b}m^x_{32}+\frac{a^2}{c}m^y_{33}=0\\
-cm^y_{12}+(c+\frac{bc}{a})m^z_{13}+\frac{ac}{b}m^y_{21}-(\frac{bc}{a}+\frac{ac}{b})m^z_{31}=0\\
-bm^x_{11}+bm^y_{21}-\frac{ac^2}{b^2}m^z_{22}-\frac{ac}{b}m^x_{23}+bm^z_{31}+(c+\frac{a^2c}{b^2})m^x_{32}-\frac{b^2}{c}m^y_{33}=0\\
\frac{c^2}{a}m^z_{13}+(c-\frac{a^2c}{b^2})m^z_{22}-(b+\frac{a^2}{b})m^x_{23}-\frac{c^2}{a}m^z_{31}+(a+\frac{a^3}{b^2})m^x_{32}=0.\\
\end{cases}
\end{align}
\end{small}
One sees that (\ref{typeone}) is a system of linear equations with variables
$m^z_{11}$, $m^x_{12}$, $m^y_{13}$, $m^x_{21}$, $m^y_{22}$, $m^z_{23}$, $m^y_{31}$, $m^z_{32}$ and $m^x_{33}$.
Its solution is  either \begin{align*}
m^x_{12}=m^x_{21}=m^z_{23}=m^z_{32}=\frac{1}{2}m^y_{22}, ~m^z_{11}=m^y_{13}=m^y_{31}=m^x_{33}=0,
\end{align*}
or
\begin{align*}
m^x_{12}=\frac{a}{c}m^x_{33},~ m^x_{21}=\frac{b}{c}m^x_{33}, ~m^z_{11}=m^y_{13}=m^y_{22}=m^z_{23}=m^y_{31}=m^z_{32}=0.
\end{align*}
  Similarly, (\ref{typetwo}) is a system of linear equations with variables  $m^y_{11}$, $m^z_{12}$, $m^x_{13}$, $m^z_{21}$, $m^x_{22}$, $m^y_{23}$, $m^x_{31}$, $m^y_{32}$ and $m^z_{33}$. And $(\ref{typetwo})$ is equivalent to \begin{align*}
m^z_{12}=\frac{a}{c}m^z_{33}, ~m^z_{21}=\frac{b}{c}m^z_{33},~m^y_{11}=m^x_{13}=m^x_{22}=m^y_{23}=m^x_{31}=m^y_{32}=0.
\end{align*}
The last system of linear equations (\ref{typethree}) has variables
 $m^x_{11}$, $m^y_{12}$, $m^z_{13}$, $m^y_{21}$, $m^z_{22}$, $m^x_{23}$, $m^z_{31}$, $m^x_{32}$ and $m^y_{33}$. By computations, its solution is
either
\begin{align*}
m^y_{12}=m^z_{13}=m^y_{21}=m^z_{31}=\frac{1}{2}m^x_{11},~ m^z_{22}=m^x_{23}=m^x_{32}=m^y_{33}=0,
\end{align*}
or
\begin{align*}
m^y_{12}=\frac{a}{c}m^y_{33},~ m^y_{21}=\frac{b}{c}m^y_{33}, ~ m^x_{11}=m^z_{13}=m^z_{22}=m^x_{23}=m^z_{31}=m^x_{32}=0.
\end{align*}
Therefore, $\mathrm{Eq}(1)$, $\mathrm{Eq}(2)$ and $\mathrm{Eq}(3)$ implies one of the following systems of equations:
\begin{small}
\begin{align}\label{v1}
\begin{cases}
m^x_{12}=m^x_{21}=m^z_{23}=m^z_{32}=\frac{1}{2}m^y_{22}, ~m^z_{11}=m^y_{13}=m^y_{31}=m^x_{33}=0\\
m^z_{12}=\frac{a}{c}m^z_{33}, ~m^z_{21}=\frac{b}{c}m^z_{33},  ~m^y_{11}=m^x_{13}=m^x_{22}=m^y_{23}=m^x_{31}=m^y_{32}=0\\
m^y_{12}=m^z_{13}=m^y_{21}=m^z_{31}=\frac{1}{2}m^x_{11}, ~m^z_{22}=m^x_{23}=m^x_{32}=m^y_{33}=0
\end{cases}
\end{align}
\begin{align}\label{v2}
\begin{cases}
m^x_{12}=\frac{a}{c}m^x_{33},~ m^x_{21}=\frac{b}{c}m^x_{33}, ~m^z_{11}=m^y_{13}=m^y_{22}=m^z_{23}=m^y_{31}=m^z_{32}=0\\
m^z_{12}=\frac{a}{c}m^z_{33}, ~m^z_{21}=\frac{b}{c}m^z_{33},  ~m^y_{11}=m^x_{13}=m^x_{22}=m^y_{23}=m^x_{31}=m^y_{32}=0\\
m^y_{12}=m^z_{13}=m^y_{21}=m^z_{31}=\frac{1}{2}m^x_{11}, ~m^z_{22}=m^x_{23}=m^x_{32}=m^y_{33}=0
\end{cases}
\end{align}

\begin{align}\label{v3}
\begin{cases}
m^x_{12}=m^x_{21}=m^z_{23}=m^z_{32}=\frac{1}{2}m^y_{22}, ~m^z_{11}=m^y_{13}=m^y_{31}=m^x_{33}=0\\
m^z_{12}=\frac{a}{c}m^z_{33}, ~m^z_{21}=\frac{b}{c}m^z_{33},  ~m^y_{11}=m^x_{13}=m^x_{22}=m^y_{23}=m^x_{31}=m^y_{32}=0\\
m^y_{12}=\frac{a}{c}m^y_{33}, ~m^y_{21}=\frac{b}{c}m^y_{33}, ~ m^x_{11}=m^z_{13}=m^z_{22}=m^x_{23}=m^z_{31}=m^x_{32}=0
\end{cases}
\end{align}
\begin{align}\label{v4}
\begin{cases}
m^x_{12}=\frac{a}{c}m^x_{33},~ m^x_{21}=\frac{b}{c}m^x_{33}, ~m^z_{11}=m^y_{13}=m^y_{22}=m^z_{23}=m^y_{31}=m^z_{32}=0\\
m^z_{12}=\frac{a}{c}m^z_{33}, ~m^z_{21}=\frac{b}{c}m^z_{33},  ~m^y_{11}=m^x_{13}=m^x_{22}=m^y_{23}=m^x_{31}=m^y_{32}=0\\
m^y_{12}=\frac{a}{c}m^y_{33}, ~m^y_{21}=\frac{b}{c}m^y_{33}, ~ m^x_{11}=m^z_{13}=m^z_{22}=m^x_{23}=m^z_{31}=m^x_{32}=0.
\end{cases}
\end{align}
\end{small}
\!Conversely, if any one of (\ref{v1}),(\ref{v2}),(\ref{v3}) and (\ref{v4}) holds, then we can get $\mathrm{Eq}(1)$, $\mathrm{Eq}(2)$ and $\mathrm{Eq}(3)$.

If $\mathrm{Eq}(1)$, $\mathrm{Eq}(2)$ and $\mathrm{Eq}(3)$ implies (\ref{v1}), then we
substitute (\ref{v1}) into the $36$ equations obtained by $\mathrm{Eq}(4)$, $\mathrm{Eq}(5)$ and $\mathrm{Eq}(6)$.  We see that those equations are equivalent to
\begin{align*}
m^k_{ij}= 0, \,\forall k\in \{x,y,z\}, \,\forall i,j \in \{1,2,3\}.
\end{align*} It indicates $\partial_{\mathcal{A}}=0$.

If $\mathrm{Eq}(1)$, $\mathrm{Eq}(2)$ and $\mathrm{Eq}(3)$ implies (\ref{v2}), then we
substitute (\ref{v2}) into the $36$ equations obtained by $\mathrm{Eq}(4)$, $\mathrm{Eq}(5)$ and $\mathrm{Eq}(6)$. We see that those equations are equivalent to
\begin{align*}
m^x_{12}=\frac{a}{c}m^x_{33},\\
m^x_{21}=\frac{b}{c}m^x_{33},\\
m^z_{12}=\frac{a}{c}m^z_{33}, \\
m^z_{21}=\frac{b}{c}m^z_{33}.
\end{align*}
Then  $\partial_{\mathcal{A}}$ is defined  by
\begin{align*}
\begin{cases}
\partial_{\mathcal{A}}(x) &= (x,y,z)\left(
                                      \begin{array}{ccc}
                                        0 & \frac{a}{c}\alpha &0 \\
                                        \frac{b}{c}\alpha & 0 & 0 \\
                                        0 & 0 & \alpha \\
                                      \end{array}
                                    \right)
\left(
                         \begin{array}{c}
                           x\\
                           y \\
                           z
                         \end{array}
                       \right) \\
\partial_{\mathcal{A}}(y) &= (x,y,z)\left(
                                      \begin{array}{ccc}
                                       0 & \frac{a}{c}\beta &0 \\
                                        \frac{b}{c}\beta & 0 & 0 \\
                                        0 & 0 & \beta \\
                                      \end{array}
                                    \right)\left(
                         \begin{array}{c}
                      x \\
                      y\\
                      z
                         \end{array}
                       \right) \\
\partial_{\mathcal{A}}(z) &= (x,y,z)\left(
                                      \begin{array}{ccc}
                                        0 & \frac{a}{c}\gamma &0 \\
                                        \frac{b}{c}\gamma & 0 & 0 \\
                                        0 & 0 & \gamma \\
                                      \end{array}
                                    \right)\left(
                         \begin{array}{c}
                           x\\
                           y \\
                           z
                         \end{array}
                       \right),
\end{cases}
\end{align*}
for some $(\alpha,\beta,\gamma)\in \Bbb{A}_k^3 $. Since
 \begin{align*}
 \begin{cases}
 ayz+bzy+cx^2=0\\
  azx+bxz+cy^2=0\\
  axy+byx+cz^2=0
  \end{cases}
   \end{align*} in $\mathcal{A}^2$, we have \begin{align*}
                     \partial_{\mathcal{A}}(x)&=(x,y,z)\left(
                                      \begin{array}{ccc}
                                         0 & \frac{a}{c}\alpha &0 \\
                                        \frac{b}{c}\alpha & 0 & 0 \\
                                        0 & 0 & \alpha \\
                                      \end{array}
                                    \right)
\left(
                         \begin{array}{c}
                           x\\
                           y \\
                           z
                         \end{array}
                       \right)\\
                       &=\frac{\alpha}{c}(byx+axy+cz^2)\\
                       &=0.
\end{align*}
Similarly, we can show that $\partial_{\mathcal{A}}(y)=\partial_{\mathcal{A}}(z)=0$. Hence $\partial_{\mathcal{A}}=0$.

If $\mathrm{Eq}(1)$, $\mathrm{Eq}(2)$ and $\mathrm{Eq}(3)$ implies (\ref{v3}), then we
substitute (\ref{v3}) into the $36$ equations obtained by $\mathrm{Eq}(4)$, $\mathrm{Eq}(5)$ and $\mathrm{Eq}(6)$.
We see that those equations are equivalent to
$$
m^z_{12}=\frac{a}{c}m^z_{33},
m^z_{21}=\frac{b}{c}m^z_{33},
m^y_{12}=\frac{a}{c}m^y_{33},
m^y_{21}=\frac{b}{c}m^y_{33}.
$$
Then  $\partial_{\mathcal{A}}$ is defined  by
\begin{align*}
\begin{cases}
\partial_{\mathcal{A}}(x) &= (x,y,z)\left(
                                      \begin{array}{ccc}
                                        0 & \frac{a}{c}\alpha &0 \\
                                        \frac{b}{c}\alpha & 0 & 0 \\
                                        0 & 0 & \alpha \\
                                      \end{array}
                                    \right)
\left(
                         \begin{array}{c}
                           x\\
                           y \\
                           z
                         \end{array}
                       \right)\\
\partial_{\mathcal{A}}(y) &= (x,y,z)\left(
                                      \begin{array}{ccc}
                                       0 & \frac{a}{c}\beta &0 \\
                                        \frac{b}{c}\beta & 0 & 0 \\
                                        0 & 0 & \beta \\
                                      \end{array}
                                    \right)\left(
                         \begin{array}{c}
                      x \\
                      y\\
                      z
                         \end{array}
                       \right) \\
\partial_{\mathcal{A}}(z) &= (x,y,z)\left(
                                      \begin{array}{ccc}
                                        0 & \frac{a}{c}\gamma &0 \\
                                        \frac{b}{c}\gamma & 0 & 0 \\
                                        0 & 0 & \gamma \\
                                      \end{array}
                                    \right)\left(
                         \begin{array}{c}
                           x\\
                           y \\
                           z
                         \end{array}
                       \right),
\end{cases}
\end{align*}
for some $(\alpha,\beta,\gamma)\in \Bbb{A}_k^3 $. As above, we can show that $\partial_{\mathcal{A}}=0$.

If $\mathrm{Eq}(1)$, $\mathrm{Eq}(2)$ and $\mathrm{Eq}(3)$ implies (\ref{v4}), then we
substitute (\ref{v4}) into the $36$ equations obtained by $\mathrm{Eq}(4)$, $\mathrm{Eq}(5)$ and $\mathrm{Eq}(6)$.
We see that those equations are equivalent to $$
m^x_{12}=\frac{a}{c}m^x_{33}, m^x_{21}=\frac{b}{c}m^x_{33},
m^z_{12}=\frac{a}{c}m^z_{33}, m^z_{21}=\frac{b}{c}m^z_{33},
m^y_{12}=\frac{a}{c}m^y_{33}, m^y_{21}=\frac{b}{c}m^y_{33}.
$$
Then  $\partial_{\mathcal{A}}$ is defined  by
\begin{align*}
\begin{cases}
\partial_{\mathcal{A}}(x) &= (x,y,z)\left(
                                      \begin{array}{ccc}
                                        0 & \frac{a}{c}\alpha &0 \\
                                        \frac{b}{c}\alpha & 0 & 0 \\
                                        0 & 0 & \alpha \\
                                      \end{array}
                                    \right)
\left(
                         \begin{array}{c}
                           x\\
                           y \\
                           z
                         \end{array}
                       \right)\\
\partial_{\mathcal{A}}(y) &= (x,y,z)\left(
                                      \begin{array}{ccc}
                                       0 & \frac{a}{c}\beta &0 \\
                                        \frac{b}{c}\beta & 0 & 0 \\
                                        0 & 0 & \beta \\
                                      \end{array}
                                    \right)\left(
                         \begin{array}{c}
                      x \\
                      y\\
                      z
                         \end{array}
                       \right)\\
\partial_{\mathcal{A}}(z) &= (x,y,z)\left(
                                      \begin{array}{ccc}
                                        0 & \frac{a}{c}\gamma &0 \\
                                        \frac{b}{c}\gamma & 0 & 0 \\
                                        0 & 0 & \gamma \\
                                      \end{array}
                                    \right)\left(
                         \begin{array}{c}
                           x\\
                           y \\
                           z
                         \end{array}
                       \right),
\end{cases}
\end{align*}
for some $(\alpha,\beta,\gamma)\in \Bbb{A}_k^3 $. As above, we can get $\partial_{\mathcal{A}}=0$.
By the discussion above, we can reach the conclusion that $\partial_{\mathcal{A}}=0$ in Case $3$.

In Case $4$, we have $c=0, a,b\in k^{\times}$. One sees that $S_{a,b,c}^3$ has a $k$-linear basis
$$\{x^3,x^2y,x^2z,xy^2,xyz,xz^2,y^3,y^2z,yz^2,z^3 \}.$$ By computations, $\mathrm{Eq}(1), \mathrm{Eq}(2)$ and $\mathrm{Eq}(3)$ are equivalent to
\begin{align}\label{f3}
\begin{cases}
(b-\frac{a^3}{b^2})m_{11}^z=0\\
(a-\frac{b^3}{a^2})m_{11}^y=0\\
(b+\frac{a^2}{b})m_{12}^z-(a+\frac{a^3}{b^2})m_{21}^z=0\\
(a-b)m_{12}^y-(a-\frac{a^2}{b})m_{13}^z-(a-b)m_{31}^z+(a-\frac{a^2}{b})m_{21}^y=0\\
(a+\frac{b^2}{a})m_{13}^y-(b+\frac{b^3}{a^2})m_{31}^y=0\\
(b-a)m_{22}^z=0\\
(a-\frac{a^2}{b})m_{22}^y-2am_{23}^z+\frac{2a^2}{b}m_{32}^z=0\\
2am_{23}^y-\frac{2a^2}{b}m_{32}^y+(\frac{a^2}{b}-a)m_{33}^z=0\\
m_{33}^y(a-b)=0\\
(a-b)m_{11}^z=0\\
-(b+\frac{a^2}{b})m_{12}^z+(a+\frac{a^3}{b^2})m_{21}^z=0\\
(b-\frac{b^2}{a})m_{11}^x-2bm_{13}^z+\frac{2b^2}{a}m_{31}^z=0\\
(-b+\frac{a^3}{b^2})m_{22}^z=0\\
(b-a)m_{12}^x-(b-a)m_{23}^z+(a-\frac{a^2}{b})m_{32}^z+(\frac{a^2}{b}-a)m_{21}^x=0\\
2bm_{13}^x+(\frac{b^2}{a}-b)m_{33}^z-\frac{2b^2}{a}m_{31}^x=0\\
(b-\frac{a^3}{b^2})m_{22}^x=0\\
(b+\frac{a^2}{b})m_{23}^x-(a+\frac{a^3}{b^2})m_{32}^x=0\\
(b-a)m_{33}^x=0\\
(b-a)m_{11}^y=0\\
(a-\frac{a^2}{b})m_{11}^x-2am_{12}^y+\frac{2a^2}{b}m_{21}^y=0\\
-(a+\frac{b^2}{a})m_{13}^y+(b+\frac{b^3}{a^2})m_{31}^y=0\\
2am_{12}^x+(\frac{a^2}{b}-a)m_{22}^y-\frac{2a^2}{b}m_{21}^x=0\\
(b-a)m_{23}^y+(a-\frac{a^2}{b})m_{13}^x+(\frac{a^2}{b}-a)m_{32}^y+(a-b)m_{31}^x=0\\
(-a+\frac{b^3}{a^2})m_{33}^y=0\\
(a-b)m_{22}^x=0\\
-(b+\frac{a^2}{b})m_{23}^x+(a+\frac{a^3}{b^2})m_{32}^x=0\\
(\frac{a^3}{b^2}-b)m_{33}^x=0.
\end{cases}
\end{align}
If $a\neq b$, then (\ref{f3}) is equivalent to
\begin{align}\label{f3neq}
\begin{cases}
m_{22}^x=m_{33}^x=m_{11}^y=m_{33}^y=m_{11}^z=m_{22}^z=0\\
m_{11}^x=\frac{b}{b-a}(2m_{31}^z-\frac{2a}{b}m_{13}^z)=\frac{b}{b-a}(2m_{12}^y-\frac{2a}{b}m_{21}^y)\\
m_{22}^y=\frac{b}{b-a}(2m_{23}^z-\frac{2a}{b}m_{32}^z)=\frac{b}{b-a}(2m_{12}^x-\frac{2a}{b}m_{21}^x)\\
m_{33}^z=\frac{b}{b-a}(2m_{23}^y-\frac{2a}{b}m_{32}^y)=\frac{b}{b-a}(2m_{31}^x-\frac{2a}{b}m_{13}^x)\\
m_{23}^x=\frac{a}{b}m_{32}^x\\
m_{31}^y=\frac{a}{b}m_{13}^y\\
m_{12}^z=\frac{a}{b}m_{21}^z.
\end{cases}
\end{align}
Substituting (\ref{f3neq}) into the $30$ equations obtained by $\mathrm{Eq}(4)$, $\mathrm{Eq}(5)$ and $\mathrm{Eq}(6)$, we see that those equations are equivalent to
\begin{align*}
\begin{cases}
(a+b)m_{11}^xm_{22}^y=0\\
(a+b)m_{11}^xm_{33}^x=0\\
(a+b)(m_{22}^y)^2=0\\
(a+b)(m_{33}^z)^2=0\\
(a+b)(m_{11}^x)^2=0\\
(a+b)m_{11}^xm_{22}^y=0\\
(a+b)m_{22}^ym_{33}^z=0\\
(a+b)(m_{33}^z)^2=0\\
(a+b)(m_{11}^x)^2=0\\
(a+b)m_{11}^xm_{33}^z=0\\
(a+b)(m_{22}^y)^2=0\\
(a+b)m_{22}^ym_{33}^z=0
\end{cases} \Leftrightarrow \quad  a=-b \quad \text{or}\quad \begin{cases}
a\neq -b\\
m_{11}^x=m_{22}^y=m_{33}^z=0.
\end{cases}
\end{align*}
Hence the equations $\mathrm{Eq}(1)\sim \mathrm{Eq}(6)$ are equivalent to
\begin{align*}
\begin{cases}
m_{22}^x=m_{33}^x=m_{11}^y=m_{33}^y=m_{11}^z=m_{22}^z=0\\
m_{11}^x=m_{31}^z+m_{13}^z=m_{12}^y+m_{21}^y\\
m_{22}^y=m_{23}^z+m_{32}^z=m_{12}^x+m_{21}^x\\
m_{33}^z=m_{23}^y+m_{32}^y=m_{31}^x+m_{13}^x\\
m_{23}^x=-m_{32}^x\\
m_{31}^y=-m_{13}^y\\
m_{12}^z=-m_{21}^z
\end{cases} \text{when}\quad a=-b\neq 0, \quad \text{and}\quad
\end{align*}
 they are equivalent to
\begin{align*}
\begin{cases}
m_{ii}^x=m_{ii}^y=m_{ii}^z=0,\forall i\in \{1,2,3\}\\
m_{12}^x=\frac{a}{b}m_{21}^x\\
m_{31}^x=\frac{a}{b}m_{13}^x\\
m_{12}^y=\frac{a}{b}m_{21}^y\\
m_{23}^y=\frac{a}{b}m_{32}^y\\
m_{23}^z=\frac{a}{b}m_{32}^z\\
m_{31}^z=\frac{a}{b}m_{13}^z\\
m_{23}^x=\frac{a}{b}m_{32}^x\\
m_{31}^y=\frac{a}{b}m_{13}^y\\
m_{12}^z=\frac{a}{b}m_{21}^z
\end{cases}\text{when} \quad a,b\in k^{\times}, a^2\neq b^2.
\end{align*}
Now, let consider the case $a=b$.  In this case,  (\ref{f3}) is equivalent to
\begin{align}\label{f3eq}
\begin{cases}
m_{12}^x=m_{21}^x\\
m_{13}^x=m_{31}^x\\
m_{23}^x=m_{32}^x\\
m_{12}^y=m_{21}^y\\
m_{13}^y=m_{31}^y\\
m_{23}^y=m_{32}^y\\
m_{12}^z=m_{21}^z\\
m_{13}^z=m_{31}^z\\
m_{23}^z=m_{32}^z.
\end{cases}
\end{align}
Substituting (\ref{f3eq}) into the $30$ equations obtained by $\mathrm{Eq}(4)$, $\mathrm{Eq}(5)$ and $\mathrm{Eq}(6)$, one sees that all those equations hold.
Therefore, the equations $\mathrm{Eq}(1)\sim \mathrm{Eq}(6)$ are equivalent to (\ref{f3eq}).

By the discussion above, we can reach the following conclusions:

(i) If $a,b\in k^{\times}, a^2\neq b^2$ and $c=0$, then  $\partial_{\mathcal{A}}$ is defined  by
\begin{align*}
\begin{cases}
\partial_{\mathcal{A}}(x) &= (x,y,z)\left(
                                      \begin{array}{ccc}
                                        0 & \frac{a}{b}\alpha_1 &\alpha_2\\
                                        \alpha_1 & 0 & \frac{a}{b}\alpha_3 \\
                                         \frac{a}{b}\alpha_2  & \alpha_3 & 0 \\
                                      \end{array}
                                    \right)
\left(
                         \begin{array}{c}
                           x\\
                           y \\
                           z
                         \end{array}
                       \right) \\
\partial_{\mathcal{A}}(y) &= (x,y,z)\left(
                                      \begin{array}{ccc}
                                        0 & \frac{a}{b}\beta_1 &\beta_2\\
                                        \beta_1 & 0 & \frac{a}{b}\beta_3 \\
                                         \frac{a}{b}\beta_2 & \beta_3 & 0 \\
                                      \end{array}
                                    \right)\left(
                         \begin{array}{c}
                      x \\
                      y\\
                      z
                         \end{array}
                       \right) \\
\partial_{\mathcal{A}}(z) &= (x,y,z)\left(
                                      \begin{array}{ccc}
                                        0 & \frac{a}{b}\gamma_1 &\gamma_2\\
                                        \gamma_1& 0 & \frac{a}{b}\gamma_3 \\
                                         \frac{a}{b}\gamma_2  & \gamma_3 & 0 \\
                                      \end{array}
                                    \right)\left(
                         \begin{array}{c}
                           x\\
                           y \\
                           z
                         \end{array}
                       \right),
\end{cases}
\end{align*}
for some $(\alpha_1,\alpha_2,\alpha_3,\beta_1,\beta_2,\beta_3, \gamma_1, \gamma_2,\gamma_3)\in \Bbb{A}_k^9 $.  Since
 $$ ayz+bzy=azx+bxz=axy+byx=0$$ in $\mathcal{A}^2$, we have \begin{align*}
                     \partial_{\mathcal{A}}(x)&=(x,y,z)\left(
                                      \begin{array}{ccc}
                                        0 & \frac{a}{b}\alpha_1 &\alpha_2\\
                                        \alpha_1 & 0 & \frac{a}{b}\alpha_3 \\
                                         \frac{a}{b}\alpha_2  & \alpha_3 & 0 \\
                                      \end{array}
                                    \right)
\left(
                         \begin{array}{c}
                           x\\
                           y \\
                           z
                         \end{array}
                       \right)\\
                       &=\alpha_1yx+\frac{a}{b}\alpha_1xy+\alpha_2xz+\frac{a}{b}\alpha_2zx+\alpha_3zy+\frac{a}{b}\alpha_3yz\\
                       &=0.
\end{align*}
Similarly, we can show that $  \partial_{\mathcal{A}}(y)=\partial_{\mathcal{A}}(z)=0$. Hence $\partial_{\mathcal{A}}=0$.

(ii)If $a=-b\in k^{\times}, c=0$, then
 $\partial_{\mathcal{A}}$ is defined  by
\begin{align*}
\begin{cases}
\partial_{\mathcal{A}}(x) &= (x,y,z)\left(
                                      \begin{array}{ccc}
                                        \alpha_1 & \alpha_3 &\alpha_4\\
                                        \beta_1-\alpha_3 & 0 & \alpha_2 \\
                                         \gamma_1-\alpha_4 & -\alpha_2 & 0 \\
                                      \end{array}
                                    \right)
\left(
                         \begin{array}{c}
                           x\\
                           y \\
                           z
                         \end{array}
                       \right) \\
\partial_{\mathcal{A}}(y) &= (x,y,z)\left(
                                      \begin{array}{ccc}
                                        0 & \beta_3 &\beta_2\\
                                       \alpha_1-\beta_3 & \beta_1 & \beta_4 \\
                                       -\beta_2 & \gamma_1-\beta_4 & 0 \\
                                      \end{array}
                                    \right)\left(
                         \begin{array}{c}
                      x \\
                      y\\
                      z
                         \end{array}
                       \right) \\
\partial_{\mathcal{A}}(z) &= (x,y,z)\left(
                                      \begin{array}{ccc}
                                        0 & \gamma_2 &\gamma_3\\
                                        -\gamma_2& 0 & \gamma_4 \\
                                         \alpha_1-\gamma_3  & \beta_1-\gamma_4 & \gamma_1 \\
                                      \end{array}
                                    \right)\left(
                         \begin{array}{c}
                           x\\
                           y \\
                           z
                         \end{array}
                       \right),
\end{cases}
\end{align*}
for some $(\alpha_1,\alpha_2,\alpha_3,\alpha_4,\beta_1,\beta_2,\beta_3,\beta_4, \gamma_1, \gamma_2,\gamma_3,\gamma_4)\in \Bbb{A}_k^{12} $.
Since
 \begin{align*}
 \begin{cases}yz-zy=0\\
               zx-xz=0\\
                xy-yx=0
  \end{cases}
  \end{align*} in $\mathcal{A}^2$, we have
\begin{align*}
&\partial_{\mathcal{A}}(x) = (x,y,z)\left(
                                      \begin{array}{ccc}
                                        \alpha_1 & \alpha_3 &\alpha_4\\
                                        \beta_1-\alpha_3 & 0 & \alpha_2 \\
                                         \gamma_1-\alpha_4 & -\alpha_2 & 0 \\
                                      \end{array}
                                    \right)
\left(
                         \begin{array}{c}
                           x\\
                           y \\
                           z
                         \end{array}
                       \right)\\
&=\alpha_1x^2+\alpha_3xy+(\beta_1-\alpha_3)yx+\alpha_4xz+(\gamma_1-\alpha_4)zx+\alpha_2yz-\alpha_2zy\\
&=\alpha_1x^2+\beta_1yx+\gamma_1zx=\alpha_1x^2+\beta_1xy+\gamma_1xz
\end{align*}
Similarly, we can show that
\begin{align*}
\partial_{\mathcal{A}}(y)=\alpha_1yx+\beta_1y^2+\gamma_1yz\\
\partial_{\mathcal{A}}(z)=\alpha_1xz+\beta_1yz +\gamma_1z^3.
\end{align*}
Let $\alpha=\alpha_1,\beta=\beta_1$ and $\gamma= \gamma_1$. Then $\partial_{\mathcal{A}}$ is defined by
\begin{align*}
\begin{cases}
\partial_{\mathcal{A}}(x)=\alpha x^2+\beta xy +\gamma xz\\
\partial_{\mathcal{A}}(y)=\alpha yx +\beta y^2+\gamma yz\\
\partial_{\mathcal{A}}(z)=\alpha xz +\beta yz +\gamma z^2, (\alpha,\beta,\gamma)\in \Bbb{A}_k^3.
\end{cases}
\end{align*}

(iii)If $a=b\in k^{\times}, c=0$, then
 $\partial_{\mathcal{A}}$ is defined by
 \begin{align*}
 \begin{cases}
\partial_{\mathcal{A}}(x) = (x,y,z)M^x
\left(
                         \begin{array}{c}
                           x\\
                           y \\
                           z
                         \end{array}
                       \right) \\
\partial_{\mathcal{A}}(y) = (x,y,z)M^y\left(
                         \begin{array}{c}
                      x \\
                      y\\
                      z
                         \end{array}
                       \right)\\
\partial_{\mathcal{A}}(z) = (x,y,z)M^z\left(
                         \begin{array}{c}
                           x\\
                           y \\
                           z
                         \end{array}
                       \right)
\end{cases},
\end{align*}
where $M^x=(m_{ij}^x)_{3\times 3}, M^y=(m_{ij}^y)_{3\times 3}$ and $M^z=(m_{ij}^z)_{3\times 3}$ are $3\times 3$ symmetric matrixes. Since
 $$ yz+zy=zx+xz=xy+yx=0$$ in $\mathcal{A}^2$, we have
                     \begin{align*}
&\quad \partial_{\mathcal{A}}(x) = (x,y,z)M^x
\left(
                         \begin{array}{c}
                           x\\
                           y \\
                           z
                         \end{array}
                       \right)\\
&=m_{11}^xx^2+ m_{12}^xxy+m_{21}^xyx+m_{22}^xy^2+m_{13}^xxz+m^x_{31}zx+m_{23}^xyz+m^x_{32}zy+m^x_{33}z^2\\
&=m_{11}^xx^2+m_{22}^xy^2+m^x_{33}z^2.
\end{align*}
Similarly, we can show that
$$ \partial_{\mathcal{A}}(y)
=m_{11}^yx^2+m_{22}^yy^2+m^y_{33}z^2$$
  and
   $$ \partial_{\mathcal{A}}(z)=m_{11}^zx^2+m_{22}^zy^2+m^z_{33}z^2.$$  Let $m_{1i}=m_{ii}^x, m_{2i}=m_{ii}^y$ and $m_{3i}=m_{ii}^z$, $i=1,2,3$.
Then \begin{align*}
\left(
                         \begin{array}{c}
                           \partial_{\mathcal{A}}(x)\\
                           \partial_{\mathcal{A}}(y)\\
                           \partial_{\mathcal{A}}(z)
                         \end{array}
                       \right)=M\left(
                         \begin{array}{c}
                           x^2\\
                           y^2\\
                           z^2
                         \end{array}
                       \right).
\end{align*}

\end{proof}

\begin{rem}\label{poly}
  When $a=b$ and $c=0$, the $3$-dimensional DG Sklyanin algebra $\mathcal{A}$ in Theorem \ref{diffstr} is just the DG algebra $\mathcal{A}_{\mathcal{O}_{-1}(k^3)}(M)$ in \cite{MWZ}. Note that Theorem \ref{diffstr} $(3)$ coincides with \cite[Proposition 2.1]{MWZ}.
\end{rem}

\section{Homological properties}
In this section, we study the homological properties of $3$-dimensional DG Sklyanin algebras.
Let $\mathcal{A}$ be a $3$-dimensional DG Sklyanin algebra with $\mathcal{A}^{\#}=S_{a,b,c}$, $(a,b,c)\in \Bbb{P}_k^2-\mathfrak{D}$.
By the differential structure classified in Theorem \ref{diffstr}, we can divide it into the following three cases:
 $$ \text{Case}\,1: |a|\neq |b| \,\,\text{or}\,\, c\neq 0; \,\,\text{Case}\,2: a=-b, c=0;\,\, \text{Case}\,3: a=b,c=0.$$

 \subsection{Case $1$} In this case, we have $\partial_{\mathcal{A}}=0$ and hence $H(\mathcal{A})=\mathcal{A}^{\#}=S_{a,b,c}$. The Calabi-Yau property of $\mathcal{A}$ is immediate from the following lemma.
\begin{lem}\cite[Proposition 3.2]{MYY}\label{imlem}
Let $\mathcal{A}$ be a connected cochain DG algebra such that $$H(\mathcal{A})=k\langle \lceil x\rceil,\lceil y\rceil,\lceil z\rceil\rangle/\left(\begin{array}{ccc}
                                 a\lceil y\rceil \lceil z\rceil +b\lceil z\rceil \lceil y\rceil + c\lceil x\rceil^{2}\\
                               a\lceil z\rceil \lceil x\rceil +b\lceil x\rceil \lceil z\rceil + c\lceil y\rceil^{2} \\
                                 a\lceil x\rceil \lceil y\rceil +b\lceil y\rceil \lceil x\rceil+ c\lceil z\rceil^{2}
                                 \end {array}\right),$$
where $(a,b,c)\in \Bbb{P}^2_k-\mathfrak{D}$ and $x,y,z\in \mathrm{ker}(\partial_{\mathcal{A}}^1)$. Then $\mathcal{A}$ is a Calabi-Yau DG algebra.
\end{lem}
Note that $H(\mathcal{A})$ in Lemma \ref{imlem} is a Koszul graded algebra. Thus the DG algebra $\mathcal{A}$ in Lemma \ref{imlem} is Koszul by
 \cite[Proposition 2.3]{HW}. By Lemma \ref{imlem}, we show the the following proposition.

\begin{prop}\label{firstcase}
Let $\mathcal{A}$ be a $3$-dimensional DG Sklyanin algebra with $\mathcal{A}^{\#}=S_{a,b,c}$, $(a,b,c)\in \Bbb{P}_k^2-\mathfrak{D}$. If we have either $|a|\neq |b|$ or $c\neq 0$,
then $\mathcal{A}$ is a Koszul Calabi-Yau DG algebra with zero differential.
\end{prop}

\subsection{Case $2$} In this case, $\partial_{\mathcal{A}}$ is defined by
\begin{align*}
\begin{cases}
\partial_{\mathcal{A}}(x)=\alpha x^2+\beta xy +\gamma xz\\
\partial_{\mathcal{A}}(y)=\alpha yx +\beta y^2+\gamma yz\\
\partial_{\mathcal{A}}(z)=\alpha xz +\beta yz +\gamma z^2,\, \text{for some}\,(\alpha,\beta,\gamma)\in \Bbb{A}_k^3.
\end{cases}
\end{align*}
If $(\alpha, \beta, \gamma)=(0,0,0)$, then $\partial_{\mathcal{A}}=0$ and hence $H(\mathcal{A})=\mathcal{A}^{\#}=S_{a,-a,0}$ with $a\in k^{\times}$. By Lemma \ref{imlem}, $\mathcal{A}$ is a Calabi-Yau DG algebra. Since $H(\mathcal{A})$ is a
Koszul graded algebra, the DG algebra $\mathcal{A}$ is Koszul by
\cite[Proposition 2.3]{HW}.

If $(\alpha, \beta, \gamma)\in \Bbb{A}_k^3-\{(0,0,0)\}$, then $\partial_{\mathcal{A}}\neq 0$.  We want to study the homological properties of $\mathcal{A}$. For this, we consider the isomorphism problem first.
Let $\mathcal{A}_1$ be the DG algebra such that $$\mathcal{A}_1^{\#}=k[x',y',z'], |x'|=|y'|=|z'|=1$$ and  $\partial_{\mathcal{A}_1}$ is defined by
\begin{align*}
\begin{cases}
\partial_{\mathcal{A}}(x')= x'^2\\
\partial_{\mathcal{A}}(y')= y'x' \\
\partial_{\mathcal{A}}(z')= x'z'.
\end{cases}
\end{align*}
We claim that $\mathcal{A}\cong \mathcal{A}_1$.
Since $(\alpha,\beta,\gamma)\neq (0,0,0)$,
we let $\alpha\neq 0$ without the loss of generality. Define a morphism $\theta: \mathcal{A}_1\to \mathcal{A}$ of graded algebras by
$$\left(\begin{array}{c}
                           \theta(x') \\
                           \theta(y') \\
                           \theta(z')
                         \end{array}
                       \right)= \left(
                                      \begin{array}{ccc}
                                        \alpha & \beta  &\gamma\\
                                        0 & 1 & 0 \\
                                         0  & 0 & 1 \\
                                      \end{array}
                                    \right)\left(
                         \begin{array}{c}
                           x \\
                           y \\
                           z
                         \end{array}
                       \right).$$
And we have $\theta\circ \partial_{\mathcal{A}_1} =\partial_{\mathcal{A}}\circ \theta$ since
\begin{align*}
\theta\circ \partial_{\mathcal{A}_1}(x')&=\theta(x'^2)=\theta(x')\theta(x')=(\alpha x+\beta y+\gamma z)(\alpha x+\beta y+\gamma z)\\
                                        &=\alpha(\alpha x^2+\beta xy+\gamma xz)+\beta(\alpha yx+\beta y^2+\gamma yz)+\gamma (\alpha xz+\beta yz+\gamma z^2)\\
                                        &=\partial_{\mathcal{A}}(\alpha x+\beta y+\gamma z)=\partial_{\mathcal{A}}\circ \theta(x'),\\
\theta\circ \partial_{\mathcal{A}_1}(y')&=\theta(y'x')=\theta(y')\theta(x')=      y(\alpha x+\beta y+\gamma z)=\partial_{\mathcal{A}}(y)=\partial_{\mathcal{A}}\circ \theta(y'),\\
\theta\circ \partial_{\mathcal{A}_1}(z')&=\theta(x'z')=\theta(x')\theta(z')=(\alpha x+\beta y+\gamma z)z=\partial_{\mathcal{A}}(z)=\partial_{\mathcal{A}}\circ \theta(z').
\end{align*}
As $\left|
                                      \begin{array}{ccc}
                                        \alpha & \beta  &\gamma\\
                                        0 & 1 & 0 \\
                                         0  & 0 & 1 \\
                                      \end{array}
                                    \right|=\alpha \neq 0$, $\theta$ is an automorphism of DG algebras.
                                    One sees that $\mathcal{A}_1$ is actually the special case of $\mathcal{A}$ when $(\alpha, \beta, \gamma)=(1,0,0)$. Hence we only need to study the homological properties of $\mathcal{A}$ when $\partial_{\mathcal{A}}$ is defined by
\begin{align*}
\begin{cases}
\partial_{\mathcal{A}}(x)= x^2\\
\partial_{\mathcal{A}}(y)= yx \\
\partial_{\mathcal{A}}(z)= xz.
\end{cases}
\end{align*}
    In this special case, we have
\begin{align*}
\partial_{\mathcal{A}}(y^2)=(yx)y-y(yx)=0\\
\partial_{\mathcal{A}}(yz)=(yx)z-y(xz)=0\\
\partial_{\mathcal{A}}(z^2)=(xz)z-z(xz)=0.
\end{align*}
So $\mathrm{im}(\partial_{\mathcal{A}}^1)=k x^2\oplus kxy\oplus kxz $ and
$$\mathrm{ker}(\partial_{\mathcal{A}}^2)=kx^2\oplus kxy\oplus kxz \oplus ky^2\oplus kyz\oplus kz^2=\mathcal{A}^2.$$
Hence $H^2(\mathcal{A})= k\lceil y^2\rceil\oplus k\lceil yz\rceil\oplus k\lceil z^2\rceil.$ We inductively assume that $\mathcal{A}^{2k}=\mathrm{ker}(\partial_{\mathcal{A}}^{2k})$ when $k\le l-1$. Since $\mathcal{A}^{2l}=\mathcal{A}^{2l-2}\cdot \mathcal{A}^2$, one sees that
$\mathcal{A}^{2l}=\mathrm{ker}(\partial_{\mathcal{A}}^{2l})$ by the Leibniz rule. Thus $\mathcal{A}^{2n}=\mathrm{ker}(\partial_{\mathcal{A}}^{2n})$ for any $n\ge 1$.
Since \begin{align*}
\begin{cases}
\partial_{\mathcal{A}}(x)= x^2\\
\partial_{\mathcal{A}}(y)= yx \\
\partial_{\mathcal{A}}(z)= xz
\end{cases}
\end{align*}
and $\mathrm{ker}(\partial_{\mathcal{A}}^{2n-2})=\mathcal{A}^{2n-2}$, it is easy to check that
$$\mathrm{im}(\partial_{\mathcal{A}}^{2n-1})=\bigoplus_{\omega_1=1}^{2n}\bigoplus_{\stackrel{\sum\limits_{j=2}^3\omega_j=2n-\omega_1}{\omega_j\ge 0, j=2,3}}k x^{\omega_1}y^{\omega_2}z^{\omega_3}.$$
 Since $$\mathcal{A}^{2n}=\bigoplus_{\stackrel{\sum\limits_{j=1}^3\omega_j=2n}{\omega_j\ge 0, j=1,\cdots,n}}k x^{\omega_1}y^{\omega_2}z^{\omega_3},$$ we have
 $$H^{2n}(\mathcal{A})=\bigoplus_{\stackrel{\sum\limits_{j=2}^n\omega_j=2n}{\omega_j\ge 0, j=2,3}}k y^{\omega_2}z^{\omega_3}.$$
For any $n\ge 2$, any cocycle element in $\mathcal{A}^{2n+1}$ can be written as
$xf+yg+zh$ for some $f,g,h\in \mathcal{A}^{2k}$. We have $\partial_{\mathcal{A}}(xf+yg+zh)=x^2f+xyg+xzh=x(xf+yg+zh)=0$.
So $xf+yg+zh=0$. Hence, $\mathrm{ker}(\partial_{\mathcal{A}}^{2k+1})=0$ and then $H^{2k+1}(\mathcal{A})=0$. Therefore,
$$H(\mathcal{A})=k[\lceil y^2\rceil,\lceil yz\rceil, \lceil z^2\rceil]/(\lceil y^2\rceil \lceil z^2\rceil-\lceil yz\rceil^2)$$
is a graded Gorenstein algebra by \cite[5.10]{Lev}. Then $\mathcal{A}$ is a Gorenstein DG algebra by \cite[Proposition 1]{Gam}. The left graded $H(\mathcal{A})$-module ${}_{H(\mathcal{A})}k$ admits a minimal free resolution:
\begin{small}
 \begin{align*}
 &\cdots \stackrel{d_n}{\to} H(A) \otimes \left(
                                                   \begin{array}{c}
                                                     ke_{(n-1)1}\\
                                                     \oplus \\
                                                     ke_{(n-1)2}\\
                                                     \end{array}
                                                 \right)\stackrel{d_{n-1}}{\to}\cdots \stackrel{d_6}{\to} H(A)\otimes \left(
                                                   \begin{array}{c}
                                                     ke_{51}\\
                                                     \oplus \\
                                                     ke_{52}\\
                                                     \end{array}
                                                 \right)  \stackrel{d_5}{\to} H(A)\otimes \left(
                                                   \begin{array}{c}
                                                     ke_{41}\\
                                                     \oplus \\
                                                     ke_{42}\\
                                                     \end{array}
                                                 \right) \\
                                                  &   \stackrel{d_4}{\to} H(A)\otimes \left(
                                                   \begin{array}{c}
                                                     ke_{t_1}\\
                                                     \oplus \\
                                                     ke_{t_2}\\
                                                     \oplus \\
                                                     ke_{t_3}
                                                   \end{array}
                                                 \right)
  \stackrel{d_3}{\to}
   H(A)\otimes \left(
                                                   \begin{array}{c}
                                                     ke_{r_1}\\
                                                     \oplus \\
                                                     ke_{r_2}\\
                                                     \oplus \\
                                                     ke_{r_3}\\
                                                     \oplus \\
                                                     ke_{r_4}
                                                   \end{array}
                                                 \right)\stackrel{d_2}{\to} H(A)\otimes \left(
                                                   \begin{array}{c}
                                                     ke_{1}\\
                                                     \oplus \\
                                                     ke_{2} \\
                                                     \oplus \\
                                                     ke_{3}
                                                   \end{array}
                                                 \right) \stackrel{d_1}{\to} H(A)\stackrel{H(\varepsilon)}{\to} k\to 0,
\end{align*}
   \end{small}
where
\begin{align*}
&d_1(e_{1})=\lceil y^2\rceil, d_1(e_{2})=\lceil yz\rceil, d_1(e_3)=\lceil z^2\rceil; \\
&d_2(e_{r_1})=\lceil y^2\rceil e_{2}-\lceil yz\rceil e_{1} \\
&d_2(e_{r_2})= \lceil y^2\rceil e_{3}-\lceil z^2\rceil e_{1}\\
&d_2(e_{r_3})=\lceil yz\rceil e_{3}-\lceil z^2\rceil e_{2} \\
&d_2(e_{r_4})=\lceil y^2\rceil e_{3}-\lceil yz\rceil e_{2}; \\
&d_3(e_{t_1})=\lceil yz\rceil e_{r_1}-\lceil y^2\rceil e_{r_2}+\lceil y^2\rceil e_{r_4}\\
&d_3(e_{t_2})=\lceil z^2\rceil e_{r_1}-\lceil yz\rceil e_{r_2}+\lceil y^2\rceil e_{r_3}\\
&d_3(e_{t_3})=\lceil z^2\rceil e_{r_1}-\lceil yz\rceil e_{r_2}+\lceil yz\rceil e_{r_4}\\
&d_4(e_{41})=\lceil yz\rceil e_{t_1}-\lceil y^2\rceil e_{t_3}\\
&d_4(e_{42})=\lceil z^2\rceil e_{t_1}-\lceil yz\rceil e_{t_3}\\
&............ \\
&d_n(e_{n1})=\lceil yz\rceil e_{(n-1)1}-\lceil y^2\rceil e_{(n-1)2}\\
&d_n(e_{n2})=\lceil z^2\rceil e_{(n-1)1}-\lceil yz\rceil e_{(n-1)2}, n\ge 5.
\end{align*}
According to the constructing procedure of Eilenberg-Moore resolution, we can construct a  semi-free resolution $F$ of the left DG $\mathcal{A}$-module $k$. The Eilenberg-Moore resolution $F$ admits a semibasis which is one to one correspondence with the free basis of the free resolution above.
We have
$$F^{\#}=\mathcal{A}^{\#}\oplus \mathcal{A}^{\#}\otimes [(\bigoplus\limits_{i=1}^3 k \Sigma e_i)\oplus (\bigoplus \limits_{j=1}^4k\Sigma^2 e_{r_j})\oplus (\bigoplus\limits_{l=1}^3k\Sigma^3 e_{t_l})\oplus (\bigoplus\limits_{s=4}^{+\infty}\bigoplus\limits_{t=1}^2k\Sigma^se_{st})]$$
 $|\Sigma e_i|=1, i\in\{1,2,3\}$, $|\Sigma^2 e_{r_j}|=2,j \in \{1,2,3,4\}$, $|\Sigma^3 e_{t_l}|=3, l\in \{1,2,3\}$ and $|\Sigma^s e_{st}|=s, s\ge 4, t\in\{1,2\}.$ From the constructing procedure of Eilenberg-Moore resolution in  \cite[P.279-280]{FHT2}, one sees that $F$ admits a semi-free filtration
 $$F(0)\subset F(1)\subset F(2)\subset \cdots \subset F(n)\subset F(n+1)\subset \cdots, $$
 where \begin{align*}
 &F(0)^{\#}=\mathcal{A}^{\#}, \\
& F(1)^{\#}=F(0)^{\#}\oplus \mathcal{A}^{\#}\otimes (\bigoplus\limits_{i=1}^3  k \Sigma e_i)\\
& F(2)^{\#}=F(1)^{\#}\oplus \mathcal{A}^{\#}\otimes (\bigoplus \limits_{j=1}^4k\Sigma^2 e_{r_j}),\\
& F(3)^{\#}=F(2)^{\#}\oplus \mathcal{A}^{\#}\otimes (\bigoplus\limits_{l=1}^3k\Sigma^3 e_{t_l})\\
&F(n)=F(3)^{\#}\oplus \mathcal{A}^{\#}\otimes (\bigoplus\limits_{s=4}^{n}\bigoplus\limits_{t=1}^2k\Sigma^se_{st}), n\ge 4.
 \end{align*}
 One sees that $F$ is minimal from the degrees of its semi-basis and the semi-free filtration above. By the minimality of $F$, we know that $\mathcal{A}$ is neither Koszul nor homologically smooth. In summary,  we obtain the following proposition.

\begin{prop}\label{polydg}
Let $\mathcal{A}$ be a connected cochain DG algebra such that $\mathcal{A}^{\#}=S_{a,-a,0}$ with $a\in k^{\times}$. Then we have the following statements.
\begin{enumerate}
\item If $\partial_{\mathcal{A}}=0$, then $\mathcal{A}$ is a Koszul and Calabi-Yau DG algebra.
\item
If $\partial_{\mathcal{A}}\neq 0$, then $\mathcal{A}$ is a Gorenstein DG algebra, but it is neither Koszul nor homologically smooth.
\end{enumerate}
\end{prop}
\subsection{Case $3$}In this cases, $\mathcal{A}^{\#}=S_{a,a,0}$ with $a\in k^{\times}$, and $\partial_{\mathcal{A}}$ is determined by a matrix $M=(m_{ij})_{3\times 3}$ such that
\begin{align*}
\left(
                         \begin{array}{c}
                           \partial_{\mathcal{A}}(x)\\
                           \partial_{\mathcal{A}}(y)\\
                           \partial_{\mathcal{A}}(z)
                         \end{array}
                       \right)=M\left(
                         \begin{array}{c}
                           x^2\\
                           y^2\\
                           z^2
                         \end{array}
                       \right)
\end{align*}
by Theorem \ref{diffstr}. It is easy for one to check that
 the DG Sklyanin algebra $\mathcal{A}$ is just the DG algebra $\mathcal{A}_{\mathcal{O}_{-1}(k^3)}(M)$ in \cite{MWZ}.
 The isomorphism problem and homological properties of $\mathcal{A}_{\mathcal{O}_{-1}(k^3)}(M)$ have been systematically studied there.
 Especially, we have the following interesting lemmas and propositions.

\begin{lem}\cite[Theorem B]{MWZ}
Let $M$ and $M'$ be two matrixes in $M_3(k)$. Then $$\mathcal{A}_{\mathcal{O}_{-1}(k^3)}(M)\cong \mathcal{A}_{\mathcal{O}_{-1}(k^3)}(M')$$ if and only if
there exists $C=(c_{ij})_{3\times 3}\in \mathrm{QPL}_3(k)$ such that
$$M'=C^{-1}M(c_{ij}^2)_{3\times 3},$$
where $\mathrm{QPL}_3(k)$ is the subgroup of $\mathrm{GL}_3(k)$ consisting of quasi-permutation matrixes.
\end{lem}

It is proved that each $\mathcal{A}_{\mathcal{O}_{-1}(k^3)}(M)$ is a Koszul. When it comes to the Calabi-Yau and homologically smooth properties,
 we have the following proposition by \cite[Theorem C]{MWZ}\label{ncy} and \cite[Theroem 5.3]{MR}.
\begin{prop}\label{quancase}
Let $\mathcal{A}$ be a connected cochain DG algebra such that $\mathcal{A}^{\#}=S_{a,a,0}$,  $a\in k^{\times}$ and $\partial_{\mathcal{A}}$ is determined by a matrix $N\in M_3(k)$ with
\begin{align*}
\left(
                         \begin{array}{c}
                           \partial_{\mathcal{A}}(x)\\
                           \partial_{\mathcal{A}}(y)\\
                           \partial_{\mathcal{A}}(z)
                         \end{array}
                       \right)=N\left(
                         \begin{array}{c}
                           x^2\\
                           y^2\\
                           z^2
                         \end{array}
                       \right).
\end{align*}
Then $\mathcal{A}$ is Koszul, and it is not Calabi-Yau (or not homologically smooth, or not Gorenstein) if and only if $\partial_{\mathcal{A}}$ satisfies the condition $(\clubsuit)$:
there exists some  $C=(c_{ij})_{3\times 3}\in \mathrm{QPL}_3(k)$ satisfying $N=C^{-1}M(c_{ij}^2)_{3\times 3}$,
where $$M=\left(
                                 \begin{array}{ccc}
                                   1 & 1 & 0 \\
                                   1 & 1 & 0 \\
                                   1 & 1 & 0 \\
                                 \end{array}
                               \right)\quad \text{or}\quad \left(
                                 \begin{array}{ccc}
                                   m_{11} & m_{12} & m_{13} \\
                                   l_1m_{11} & l_1m_{12} & l_1m_{13} \\
                                   l_2m_{11} & l_2m_{12} & l_2m_{13} \\
                                 \end{array}
                               \right)$$ with $m_{12}l_1^2+m_{13}l_2^2\neq m_{11}, l_1l_2\neq 0$ and $4m_{12}m_{13}l_1^2l_2^2= (m_{12}l_1^2+m_{13}l_2^2-m_{11})^2$.
                               For the second case,
                               neither $m_{12}m_{11}< 0$ nor $m_{13}m_{11}< 0$ will occur.
 Furthermore,
\begin{enumerate}
\item if $m_{11}=0$, then $m_{12}l_1=m_{13}l_2$ and
$\mathcal{A}_{\mathcal{O}_{-1}(k^3)}(M)$ is isomorphic to $\mathcal{A}_{\mathcal{O}_{-1}(k^3)}(X)$, where $$X=\left(
                                 \begin{array}{ccc}
                                   0 & m_{12} & m_{12} \\
                                   0 & l_1m_{12} & l_1m_{12} \\
                                   0 & l_2\sqrt{m_{12}m_{13}} & l_2\sqrt{m_{12}m_{13}} \\
                                 \end{array}
                               \right);$$
\item if $m_{11}m_{12}>0, m_{11}m_{13}>0$ then $\mathcal{A}_{\mathcal{O}_{-1}(k^3)}(M)$  is isomorphic to $\mathcal{A}_{\mathcal{O}_{-1}(k^3)}(Q)$, where $$Q=\left(
                                 \begin{array}{ccc}
                                   m_{11}\sqrt{m_{12}m_{13}} & m_{11}\sqrt{m_{12}m_{13}} & m_{11}\sqrt{m_{12}m_{13}} \\
                                   l_1m_{12}\sqrt{m_{11}m_{13}} & l_1m_{12}\sqrt{m_{11}m_{13}} & l_1m_{12}\sqrt{m_{11}m_{13}} \\
                                   l_2m_{13}\sqrt{m_{11}m_{12}} & l_2m_{13}\sqrt{m_{11}m_{12}} & l_2m_{13}\sqrt{m_{11}m_{12}} \\
                                 \end{array}
                               \right).$$
\end{enumerate}
\end{prop}
\begin{rem}
For briefness, we say that $\partial_{\mathcal{A}}$ satisfies $(\spadesuit)$ if the condition $(\clubsuit)$ doesn't holds. Note that the differential of $\mathcal{A}_{\mathcal{O}_{-1}(k^3)}(X)$ in Proposition \ref{ncy}(1) is defined by
\begin{align*}
\begin{cases}
\partial_{\mathcal{A}}(x_1)=m_{12}(x_2^2+x_3^2)\\
\partial_{\mathcal{A}}(x_2)=l_1m_{12}(x_2^2+x_3^2)\\
\partial_{\mathcal{A}}(x_3)=l_2\sqrt{m_{12}m_{13}}(x_2^2+x_3^2),
\end{cases}
\end{align*}
where $l_1m_{12}=l_2m_{13}, l_1,l_2,m_{12},m_{13}\in k^{\times}$. Let $l_1=l_2=m_{12}=m_{13}=1$. Then $X=\left(
                           \begin{array}{ccc}
                             0 & 1 & 1 \\
                             0 & 1 & 1 \\
                             0 & 1 & 1 \\
                           \end{array}
                         \right)$ and we get a simple example of $3$-dimensional DG Sklyanin algebra, which is not  homologically smooth but Koszul.
Similarly, the differential of $\mathcal{A}_{\mathcal{O}_{-1}(k^3)}(Q)$ in Proposition \ref{ncy}(2) is defined by
\begin{align*}
\begin{cases}
\partial_{\mathcal{A}}(x_1)= m_{11}\sqrt{m_{12}m_{13}}(x_1^2+x_2^2+x_3^2)\\
\partial_{\mathcal{A}}(x_2)=l_1m_{12}\sqrt{m_{11}m_{13}}(x_1^2+x_2^2+x_3^2)\\
\partial_{\mathcal{A}}(x_3)=l_2m_{13}\sqrt{m_{11}m_{12}}(x_1^2+x_2^2+x_3^2)
\end{cases}
\end{align*}
where $m_{12}m_{13},m_{11}m_{13},m_{11}m_{12}>0, l_1l_2\neq 0$ and $$4m_{12}m_{13}l_1^2l_2^2= (m_{12}l_1^2+m_{13}l_2^2-m_{11})^2.$$ For example, let $l_1=m_{11}=m_{12}=m_{13}=1, l_2=2$, then $Q=\left(
                           \begin{array}{ccc}
                             1 & 1 & 1 \\
                             1 & 1 & 1 \\
                             2 & 2 & 2 \\
                           \end{array}
                         \right)$ and $\mathcal{A}_{\mathcal{O}_{-1}(k^3)}(Q)$ is a simple example of $3$-dimensional DG Sklyanin algebra, which is not homologically smooth but Koszul.
\end{rem}

By the discussion above, one sees that almost all $3$-dimensional DG Sklyanin algebras are Calabi-Yau DG algebras except a few special cases. More precisely, we have the following theorem.
\begin{thm}\label{finthm}
Let $\mathcal{A}$ be a $3$-dimensional DG Sklyanin algebra with $\mathcal{A}^{\#}=S_{a,b,c}$.  Then $\mathcal{A}$ is not Calabi-Yau (or not homologically smooth) if and only if one of the two following conditions holds.
\begin{enumerate}
\item $a=-b, c=0$ and $\partial_{\mathcal{A}}\neq 0$;
\item $a=b, c=0$ and $\partial_{\mathcal{A}}$ satisfies $(\clubsuit)$.
\end{enumerate}
\end{thm}
\begin{proof}
The `if' part is trivial by Proposition \ref{polydg} and Proposition \ref{ncy}. We only need to show the `only if' part.
If the $3$-dimensional DG Sklyanin algebra $\mathcal{A}$ is not Calabi-Yau, then $a=b, c=0$ or $a=-b,c=0$ by Proposition \ref{firstcase}.
For the case that $a=-b,c=0$, we have $\partial_{\mathcal{A}}=0$ by Proposition \ref{polydg}. When $a=b, c=0$,
 Theorem \ref{diffstr} indicates that $\partial_{\mathcal{A}}$ is determined by a matrix $N\in M_3(k)$ such that
\begin{align*}
\left(
                         \begin{array}{c}
                           \partial_{\mathcal{A}}(x)\\
                           \partial_{\mathcal{A}}(y)\\
                           \partial_{\mathcal{A}}(z)
                         \end{array}
                       \right)=N\left(
                         \begin{array}{c}
                           x^2\\
                           y^2\\
                           z^2
                         \end{array}
                       \right).
\end{align*}
By Proposition \ref{ncy}, there exists some  $C=(c_{ij})_{3\times 3}\in \mathrm{QPL}_3(k)$ satisfying $N=C^{-1}M(c_{ij}^2)_{3\times 3}$,
where $$
 M=\left(
                                 \begin{array}{ccc}
                                   1 & 1 & 0 \\
                                   1 & 1 & 0 \\
                                   1 & 1 & 0 \\
                                 \end{array}
                               \right)\quad \text{or}\quad
M=\left(
                                 \begin{array}{ccc}
                                   m_{11} & m_{12} & m_{13} \\
                                   l_1m_{11} & l_1m_{12} & l_1m_{13} \\
                                   l_2m_{11} & l_2m_{12} & l_2m_{13} \\
                                 \end{array}
                               \right)$$ with $m_{12}l_1^2+m_{13}l_2^2\neq m_{11}, l_1l_2\neq 0$ and $4m_{12}m_{13}l_1^2l_2^2= (m_{12}l_1^2+m_{13}l_2^2-m_{11})^2$.
\end{proof}

It is well-known that Calabi-Yau property of a connected cochain DG algebra implies its Gorenstein property and homologically smoothness.
We can finish this section with the following tabular as a summary on the homological properties of $3$-dimensional DG Sklyanin algebras.
\\
\begin{tabular}{l|llll}\hline
\backslashbox{cases}{properties}& Koszul & Gorenstein & homologically smooth & Calabi-Yau \\
 \hline
 $|a|\neq |b|$ or $c\neq 0$ & \Checkmark & \Checkmark &\Checkmark & \Checkmark \\
 $a=-b,c=0, \partial_{\mathcal{A}}=0$ & \Checkmark & \Checkmark &\Checkmark& \Checkmark  \\
 $a=-b,c=0,\partial_{\mathcal{A}}\neq 0$ & \XSolid &\Checkmark &\XSolid & \XSolid\\
 $a=b,c=0$ and $(\clubsuit)$ & \Checkmark & \XSolid & \XSolid &\XSolid \\
 $a=b,c=0$ and $(\spadesuit)$ & \Checkmark &\Checkmark &\Checkmark &\Checkmark
 \\
 \hline
\end{tabular}

\section*{Acknowledgments}
X.-F. Mao was supported by NSFC (Grant No.11871326). X.-T. Wang was supported by Simons Foundation
Program: Mathematics and Physical Sciences-Collaboration Grants for Mathematician (Award No.688403). The authors thank Professor James Zhang for his useful suggestions and comments on this paper.

\end{document}